\documentclass[11pt]{article} 

\usepackage[utf8]{inputenc} 

\usepackage[margin=1in]{geometry} 
\usepackage{graphicx} 

\usepackage{booktabs} 
\usepackage{array} 
\usepackage{paralist} 
\usepackage{verbatim} 
\usepackage{mathrsfs}
\usepackage{amssymb}
\usepackage{amsthm}
\usepackage{amsmath,amsfonts,amssymb,esint}
\usepackage{graphics}
\usepackage{enumerate}
\usepackage{mathtools}
\usepackage{xfrac}
\usepackage{bbm}
\usepackage{subcaption}
\usepackage{times}

\usepackage[usenames,dvipsnames]{xcolor}
\usepackage[colorlinks=true, pdfstartview=FitV, linkcolor=blue, citecolor=blue, urlcolor=blue]{hyperref}

\numberwithin{equation}{section}

\newtheorem{theorem}{Theorem}[section]

\newtheorem{proposition}[theorem]{Proposition}
\newtheorem{lemma}[theorem]{Lemma}
\newtheorem{definition}[theorem]{Definition}
\theoremstyle{definition}
\newtheorem{remark}[theorem]{Remark}

\newcommand{\norm}[1]{\left\|#1\right\|}
\newcommand{\abs}[1]{\left|#1\right|}

\newcommand*{\supp}{\ensuremath{\mathrm{supp\,}}}
\newcommand*{\Id}{\ensuremath{\mathrm{Id}}}
\renewcommand*{\div}{\ensuremath{\mathrm{div\,}}}

\newcommand*{\N}{\ensuremath{\mathbb{N}}}

\newcommand*{\T}{\ensuremath{\mathbb{T}}}

\newcommand*{\R}{\ensuremath{\mathbb{R}}}

\newcommand{\eps}{\varepsilon}
\newcommand{\OO}{\mathcal O}

\newcommand{\EE}{\mathcal E}
\newcommand{\HH}{\mathcal H}

\newcommand{\RR}{\mathring R}

\newcommand{\RSZ}{\mathcal R}

\newcommand*{\curl}{\ensuremath{\mathrm{curl\,}}}

\newcommand{\les}{\lesssim}

\title{\bf Weak solutions of ideal MHD which\\  do not conserve magnetic helicity}

\author{Rajendra Beekie
\thanks{Courant Institute of Mathematical Sciences,  New York University.
{\footnotesize \href{mailto:beekie@cims.nyu.edu}{beekie@cims.nyu.edu}.}
}
\and
Tristan Buckmaster
\thanks{Department of Mathematics, Princeton University.
{\footnotesize \href{mailto:buckmaster@math.princeton.edu}{buckmaster@math.princeton.edu}.}
}
\and 
Vlad Vicol
\thanks{Courant Institute of Mathematical Sciences,  New York University.
{\footnotesize \href{mailto:vicol@cims.nyu.edu}{vicol@cims.nyu.edu}.}
}
}
\date{ }

\usepackage[nottoc,notlot,notlof]{tocbibind}

\begin{document}

\maketitle

\begin{abstract}
We construct weak solutions to the ideal magneto-hydrodynamic (MHD) equations which have finite total energy, and whose magnetic helicity is not a constant function of time. In view of Taylor's conjecture, this proves that there exist finite energy weak solutions to ideal MHD which cannot be attained in the infinite conductivity and zero viscosity limit. Our proof is based on a Nash-type convex integration scheme with intermittent building blocks adapted to the geometry of the MHD system.
\end{abstract}

\section{Introduction}
We consider the three-dimensional incompressible {\em ideal} magneto-hydrodynamic (MHD) equations
\begin{subequations}
\label{eq:ideal:mhd}
\begin{align}
&\partial_t u + (u\cdot \nabla) u - (B \cdot \nabla) B + \nabla p  = 0 
\\
&\partial_t B + (u\cdot \nabla) B - (B\cdot \nabla) u  = 0
\\
& \div u  = \div B =  0.
\end{align}
\end{subequations}
posed on the periodic box $\T^3 = [-\pi,\pi]^3$, for the velocity field  $u : \mathbb{T}^3 \times [0,T] \to \mathbb{R}^3 $, the magnetic field $ B: \mathbb{T}^3 \times [0,T]  \to \mathbb{R}^3$, and the scalar pressure $p:\mathbb{T}^3 \times [0,T] \to \mathbb{R}$. This is the classical macroscopic model coupling Maxwell's equations to the evolution of an electrically conducting fluid/plasma~\cite{Biskamp97,Davidson01,SermangeTemam83}.

\subsection{MHD conservation laws}
The ideal MHD equations~\eqref{eq:ideal:mhd} posses a number of conservation laws, which inform the class of solutions we work with. The {\em mean} of $u$ and $B$ over $\T^3$ are conserved in time (even for weak solutions)  and thus we consider solutions of~\eqref{eq:ideal:mhd} such that $\int_{\T^3} u(x,t) dx = \int_{\T^3} B(x,t) dx = 0$. For smooth solutions of \eqref{eq:ideal:mhd} the coercive conservation law, and in fact Hamiltonian~\cite{ViskikDolzhanskii78,KhesinEtAl19}  of the system, is given by the {\em total energy} 
\[
\EE(t) = \frac 12 \int_{\T^3} |u(x,t)|^2 + |B(x,t)|^2 \, dx \, .
\] 
This motivates us to work with solutions to \eqref{eq:ideal:mhd} such that $u(\cdot,t), B(\cdot, t) \in L^2(\T^3)$ for all times $t$. At this $L^\infty_t L^2_x$ regularity level the {\em cross helicity} 
\[ 
\HH_{\omega,B} = \int_{\T^3} u(x,t) \cdot B(x,t) dx
\] 
is well-defined, and    \eqref{eq:ideal:mhd} formally  conserves the cross helicity. Lastly, we mention the conservation of the {\em magnetic helicity}~\cite{Woltjer58a,Woltjer58b,Moffatt69}, defined as 
\[
\HH_{B,B}(t) = \int_{\T^3} A(x,t) \cdot B(x,t) \, dx \, ,
\]  
where $A$ is a vector potential for $B$, i.e. $\curl A = B$. As we work on the simply connected domain $\T^3$, the value of $\HH_{B,B}(t)$ is independent of the choice of $A$. Indeed, keeping in mind the Helmholtz decomposition we note that the gradient part of $A$ is orthogonal to $B$, and thus $A$ may be chosen without loss of generality such that $\div A = 0$ and $ \int_{\mathbb{T}^3} A(x,t) dx  = 0$.
Throughout the paper we work with this representative vector potential  given by the Biot-Savart law: $A = \curl (-\Delta)^{-1} B$. This also justifies our generalized helicity notation used above: $\HH_{f,g} = \int_{\T^3} \curl (-\Delta)^{-1} f \cdot g \, dx$ (see also~\cite{KhesinEtAl19}).

We emphasize that as opposed to the total energy and cross helicity (the so-called Els\"asser energies~\cite{AluieEyink10}), the magnetic helicity lies at a negative regularity level, namely $L^\infty_t \dot{H}^{-1/2}_x$. This subtle difference points to the fact that magnetic helicity plays a special role among the conserved quantities of \eqref{eq:ideal:mhd}, a fact which is famously manifested in the context of reconnection events in magneto-hydrodynamic turbulence. While turbulent low-density plasma configurations are observed to dissipate   energy~\cite{MininniPouquet09,DallasAlexakis14}, it is commonly accepted knowledge in the plasma physics literature that the magnetic helicity is  conserved in the infinite conductivity limit. This striking phenomenon is known as {\em Taylor's conjecture}~\cite{Taylor74,Taylor86,Berger99,Escande15,Moffatt15}, and we recall in Section~\ref{sec:Taylor} its mathematical foundations~\cite{Caflisch1997,FaracoLindberg18b}. In contrast, our main result (cf.~Theorem~\ref{thm:main}) shows that there exist weak solutions of the ideal MHD equations (cf.~Definition~\ref{def:weak:sol}) whose {\em magnetic helicity is not a constant function of time}. We thus prove that the ideal-MHD-version of Taylor's conjecture is false. 

\subsection{Weak solutions and Onsager exponents for MHD}
Before stating our result precisely, we recall a number of previous works on this subject. First, we introduce the notion of  {\em weak/distributional} solutions to \eqref{eq:ideal:mhd} that we consider in this paper. We work with solutions of regularity at the level of the strongest known coercive conservation law, i.e., they have finite energy.

\begin{definition}[\bf Weak solution]
\label{def:weak:sol}
We say $(u, B) \in C((-T,T); L^2(\mathbb{T}^3))$ is a weak solution of the ideal MHD system \eqref{eq:ideal:mhd} if for any $t \in (-T, T) $ the vector fields $(u( \cdot,t ), B(\cdot,t ))$ are divergence free in the sense of distributions, they have zero mean, and \eqref{eq:ideal:mhd} holds in the sense of distributions, i.e. 
\begin{align*}
\int_{-T}^T \int_{\mathbb{T}^3}  \partial_t \psi  \cdot u + \nabla \psi  :( u \otimes u - B \otimes B) dx dt &= 0  
  \\
\int_{-T}^T \int_{\mathbb{T}^3}  \partial_t \psi  \cdot B + \nabla \psi  :( u \otimes B - B \otimes u) dx dt &= 0  
\end{align*}
hold for all divergence free test functions $\psi  \in C_0^{\infty} ((-T, T) \times \mathbb{T}^3) $.
\end{definition}

In analogy with the famed Onsager conjecture for weak solutions of the 3D Euler equations~\cite{Onsager49}, it is natural to ask the question of the minimal regularity required by weak solutions of \eqref{eq:ideal:mhd} to respect the ideal MHD conservation laws: the energy $\EE$, the cross helicity $\HH_{\omega,B}$, and the magnetic helicity $\HH_{B,B}$. Once a suitable scale of Banach spaces is fixed to measure regularity, this putative minimal regularity exponent defines a {\em critical/threshold exponent} above which all weak solutions obey the given conservation law (the rigid side), while below this exponent there exist weak solutions which violate it (the flexible side). See~\cite{Klainerman17}, where this question is posed for general nonlinear, supercritical, Hamiltonian evolution equations (3D Euler and 3D MHD being examples of such systems), \cite[Remark 1.8]{BSV16}  in the context of the SQG system, and~\cite{IsettVicol15} for more general active scalar equations.

Concerning the conservation of the $L^2_x$ quantities $\EE$ and $\HH_{\omega,B}$, similar results have been established in parallel to the rigid side of the Onsager conjecture in 3D Euler~\cite{CET94,Eyink94,CCFS08}. To see this, recall that the Els\"asser variables $z^{\pm} = u \pm B$ are incompressible and obey $\partial_t z^{\pm} + z^{\mp} \cdot \nabla z^{\pm} = -\nabla \Pi$, where $\Pi = p + b^2/2$. Using the commutator estimates of~\cite{CET94}, Caflisch-Klapper-Steele~\cite{Caflisch1997} proved the conservation of energy and cross helicity for weak solutions $(u,B) \in B^{\alpha}_{3,\infty}$ with $\alpha>1/3$. See also~\cite{KangLee07} who use the methods of~\cite{CCFS08} to reach the endpoint case $B^{1/3}_{3,c(\N)}$. 

The {\em analogy with 3D Euler spectacularly fails} when we consider the {\em flexible part of the Onsager question}, namely to construct weak solutions to \eqref{eq:ideal:mhd}, in the sense of Definition~\ref{def:weak:sol}, with  regularity below $1/3$ when measured in $L^3$, that do not conserve energy, or cross helicity. For 3D Euler the Onsager conjecture is now solved, cf.~Isett~\cite{Isett18}, and B.-De Lellis-Sz\'ekelyhidi-V.~\cite{BDLSV19} for dissipative solutions.  In contrast, for 3D MHD the only non-trivial (i.e. $B\not\equiv 0$) non-conservative example arises when one imposes a symmetry assumption which embeds the system into a $2 \frac 12$D Euler flow:  if $v = (v_1,v_2,v_3)(x_1,x_2)$ solves 3D Euler, then setting $u = (v_1,v_2,0)$ and $B= (0,0,v_3)$, the resulting $x_3$ independent functions solve the ideal MHD system. This symmetry reduced system is used by Bronzi-Lopes Filho-Nussenzveig Lopes in~\cite{BronziEtAl15} to construct an example with $\EE$ not constant. Note however that in this case both the cross helicity and the magnetic helicity vanish identically, so that $\HH_{\omega,B} = \HH_{B,B} = 0$ are conserved. Thus, {\em to date there are no known examples of non-conservative, truly 3D, weak solutions to} \eqref{eq:ideal:mhd}. The only attempt at constructing wild solutions is the work of Faraco-Lindberg~\cite{FaracoLindberg18a}, who use the ideas of De Lellis-Sz\'ekelyhidi~\cite{DLSZ09} and the Tartar framework~\cite{Tartar83} to show that there do in fact exist non-vanishing smooth  {\em strict subsolutions} of 3D ideal MHD with compact support in space-time. However, the interior of the 3D $\Lambda$-convex hull is empty, and it is not known if a convex integration approach would succeed to construct an actual weak solution, starting with this subsolution. In fact, in this same paper~\cite{FaracoLindberg18a} it is shown that ideal 2D MHD does not have weak solutions (or even subsolutions) with compact support in time and with $B\not \equiv 0$.  The emptiness of the interior of the 3D $\Lambda$-convex hull for \eqref{eq:ideal:mhd} may seem like just a technical obstacle for the flexible part of the Onsager question. There is, however, a {\em fundamental physical reason} why the construction of $L^\infty_{x,t}$ weak solutions to \eqref{eq:ideal:mhd} fails. A convex integration scheme which would produce weak solutions $(u,B)\in L^\infty_{x,t}$ such that $\EE$ and $\HH_{\omega,B}$ are non-constant, would inadvertently also show that $\HH_{B,B}$ is non-constant. This is, however, impossible: the magnetic helicity is conserved by weak solutions  under much milder assumptions.  We note that a parallel obstruction for $L^\infty_{x,t}$ convex-integration constructions occurs in the setting of the SQG equation: the kinetic energy conservation requires that the potential vorticity has $1/3$ regularity, whereas the conservation of the Hamiltonian only requires $L^3_{t,x}$ integrability~\cite{IsettVicol15,BSV16}. 

Indeed,  Caflisch-Klapper-Steele prove in~\cite{Caflisch1997} that the magnetic helicity is conserved by weak solutions of \eqref{eq:ideal:mhd} as soon as $(u,B) \in B^{\alpha}_{3,\infty}$ with $\alpha>0$. Note the considerably weaker condition  $\alpha>0$ for $\HH_{B,B}$ conservation, as opposed to $\alpha>1/3$ for $\EE$. Kang-Lee~\cite{KangLee07} and subsequently Aluie~\cite{Aluie09} and Faraco-Lindberg~\cite{FaracoLindberg18a} were able to derive the endpoint case which states that magnetic helicity is conserved as soon as $(u,B) \in L^3_{x,t}$. This discrepancy between the requirements for energy and magnetic helicity conservation is the underlying physical difficulty to our construction, known in the plasma physics community as Taylor's conjecture (discussed in Section~\ref{sec:Taylor} below). 

Whether the $ L^3_{x,t}$ regularity threshold for the conservation of $\HH_{B,B}$ is sharp remains open. As mentioned before, we do not have  examples of non-conservative solutions to \eqref{eq:ideal:mhd}. This open problem is stated explicitly in~\cite{FaracoLindberg18b}: ``{\em It is still open whether magnetic helicity is conserved if $u$ and $B$ belong to the energy space $L^\infty(0, T ; L^2(\T^3, \R^3))$}''. In this paper we answer this question in the positive, see Theorem~\ref{thm:main}.

\subsection{Taylor's conjecture}
\label{sec:Taylor}

Before turning to our main result, we briefly discuss the mathematical aspects of Taylor's conjecture, which  has interesting consequences concerning the set of weak solutions to~\eqref{eq:ideal:mhd}.

The viscous ($\nu>0)$ and resistive ($\mu>0$) MHD equations are given by
\begin{subequations}
\label{eq:MHD}
\begin{align}
&\partial_t u + (u\cdot \nabla) u - (B \cdot \nabla) B + \nabla p  = \nu \, \Delta u \\
&\partial_t B + (u\cdot \nabla) B - (B\cdot \nabla) u  = \mu \, \Delta B \\
& \div u  = \div B =  0.
\end{align}
\end{subequations}
In analogy to the 3D Navier-Stokes equation, using the  energy inequality for \eqref{eq:MHD}
\begin{equation}
    \label{eq:energy:inequality}
        \EE(t) +   \nu \int_{t_0}^t \|\nabla u(\cdot,s) \|_{L^2}^2 ds  +  \mu \int_{t_0}^t \|\nabla B(\cdot,s) \|_{L^2}^2 ds  \leq  \EE(t_0) \, ,
\end{equation}
it is classical to build a theory of Leray-Hopf weak solutions for \eqref{eq:MHD}. These are solutions with $u,B \in C^0_{w,t} L^2_x \cap L^2_t \dot{H}^1_x$ which obey \eqref{eq:energy:inequality} for a.e. $t_0\geq 0$ and all $t>t_0$. Note that the only {\em uniform in $(\nu,\mu)$ bounds} for Leray-Hopf weak solutions to \eqref{eq:MHD} are at the $L^\infty_t L^2_x$ regularity level, as in Definition~\ref{def:weak:sol}. Following \cite[Definition 1.1]{FaracoLindberg18b} we recall the definition: 

\begin{definition}[{\bf Weak ideal limit}~\cite{FaracoLindberg18b}]
Let $(\nu_j, \mu_j) \to (0,0)$ be a sequence of vanishing viscosity and resistivity. Associated to a sequence of divergence free initial data converging weakly  $(u_{0,j}, B_{0,j}) \rightharpoonup (u_{0}, B_{0})$ in $L^{2}(\mathbb{T}^3)$, let $(u_j , B_j)$ be a sequence of Leray-Hopf weak solutions of \eqref{eq:MHD}.  Any pair of functions $(u,B)$ such that $(u_j,B_j) \stackrel{\ast}{\rightharpoonup} (u,B)$ in $L^{\infty}(0,T; L^2(\mathbb{T}^3))$, are called a {\em weak ideal limit} of the sequence $(u_j, B_j)$. 
\end{definition}

Note in particular that a weak ideal limit $(u,B)$ need not be a weak solution of the ideal MHD equations \eqref{eq:ideal:mhd}. Taylor's conjecture states that weak ideal limits of Leray-Hopf weak solutions to \eqref{eq:MHD} conserve the magnetic helicity. This was proven recently in~\cite{FaracoLindberg18b}:

\begin{theorem}[{\bf Proof of Taylor's conjecture}~\cite{FaracoLindberg18b}]
\label{thm:Taylor}
Suppose $(u,B) \in L^\infty_t L^2_x$ is a weak ideal limit of a sequence of Leray-Hopf weak solutions. Then $\HH_{B,B}$ is a constant function of time. In particular, finite energy weak solutions of the ideal MHD equations~\eqref{eq:ideal:mhd} which are weak ideal limits, conserve magnetic helicity.
\end{theorem}

The proof of Theorem~\ref{thm:Taylor} given in~\cite{FaracoLindberg18b} (who also consider domains which are not simply connected) has three ingredients: Leray-Hopf weak solutions to \eqref{eq:MHD} have desirable properties which may be deduced from~\eqref{eq:energy:inequality}, the magnetic helicity is bounded as soon as $B \in L^\infty_t \dot{H}^{-1/2}_x$, and the fact $L^2\subset \dot{H}^{-1/2}$ is compact (we work with zero mean functions). We recall this argument in Appendix~\ref{app:Taylor} and note that similar proofs appear in the context of the 2D Euler equations~\cite{CLfNlS16} and of the 2D SQG equations~\cite{ConstantinIgnatovaNguyen18}.

In conclusion, we emphasize that there is a substantial {\em integrability/scaling discrepancy} between the results of~\cite{KangLee07,Aluie09,FaracoLindberg18a}, which consider the conservation of $\HH_{B,B}$ directly for weak solutions of  ideal MHD, and the result of Taylor's conjecture~\cite{FaracoLindberg18b}, which considers weak solutions to \eqref{eq:ideal:mhd} that arise as weak ideal limits from \eqref{eq:MHD}. The first set of results require $L^3_{x,t}$ integrability to guarantee that the magnetic helicity is constant in time, while the second result requires merely $L^\infty_t L^2_{x}$ integrability. Thus, there is additional hidden information in the definition of a weak ideal limit, a ghost of the energy inequality~\eqref{eq:energy:inequality}. Our goal in this paper is to show that this scaling discrepancy is real, by proving  that there exist  $L^\infty_t L^2_{x}$ weak solutions to ideal MHD which do not conserve magnetic helicity (see Section~\ref{sec:main:result} for details).

\subsection{Results and new ideas}
\label{sec:main:result}

In this paper we prove the existence of non-trivial non-conservative weak solutions to \eqref{eq:ideal:mhd} with finite kinetic energy. For clarity of the presentation, we only prove the simplest version of this statement:
\begin{theorem}[\bf Main result]
\label{thm:main}
There exists $\beta > 0$ such that the following holds. There exist     weak solutions $(u ,B) \in C([0,1], H^{\beta}) $ of  \eqref{eq:ideal:mhd}, in the sense of Definition~\ref{def:weak:sol}, which do not conserve magnetic helicity. In particular, there exist solutions as above with $2\abs{\HH_{B,B}(0)} \leq \HH_{B,B}(1)  $ and $\HH_{B,B}(1) > 0$. For these solutions the total energy $\EE$ and cross helicity $\HH_{\omega,B}$ are non-trivial non-constant functions of time.
\end{theorem}

To the best of our knowledge Theorem~\ref{thm:main} provides the {\em first example} of a non-conservative weak solution to the ideal MHD equations, for which $\EE, \HH_{\omega,B}$ and $\HH_{B,B}$ are all non-trivial. A direct consequence of our result is the non-uniqueness of weak solutions to \eqref{eq:ideal:mhd} in the sense of Definition~\ref{def:weak:sol}. In fact, at this $L^\infty_t L^{2}_x$ regularity level, Theorem~\ref{thm:main} also gives the first {\em existence result} for weak solutions to \eqref{eq:ideal:mhd}, as the usual weak-compactness methods from smooth approximations fail, for the same reasons they fail in 3D Euler. In fact, we note that in view of Theorem~\ref{thm:Taylor}, the weak solutions of 3D ideal MHD which we construct in Theorem~\ref{thm:main} {\em cannot be obtained as weak ideal limits} from   3D viscous and resistive MHD.

The regularity of the weak solutions from Theorem~\ref{thm:main} is slightly better than $C^0_t L^2_x$, as the parameter $\beta$ is very small (as in~\cite{BV}). In view of the conservation of magnetic helicity in $C^0_t L^3_x$, and of the Sobolev embedding, any construction of non-conservative weak solutions in $H^\beta$ must have $\beta < 1/2$. However, it seems that fundamentally new ideas are needed to substantially increase the value of $\beta$ in Theorem~\ref{thm:main}. Additionally, making progress towards the flexible side of an Onsager conjecture for ideal MHD, i.e. to construct weak solutions in $B^\alpha_{3,\infty}$ with $0< \alpha < 1/3$ which do not conserve total energy seems out of reach of current methods (such solutions would need to conserve magnetic helicity, but not total energy). 

The proof of Theorem~\ref{thm:main} is based on a {\em Nash-style convex integration} scheme with {\em intermittent} building blocks adapted to the specific {\em geometry of the MHD system}. For the 3D Euler equations,  Scheffer~\cite{Scheffer93} and Shnirelman~\cite{Shnirelman97} first gave examples of wild solutions in $L^2_x$, respectively $L^\infty_x$, while De Lellis-Sz\'ekelyhidi~\cite{DLSZ09} have placed these constructions in a unified mathematical framework.
Convex integration schemes based on the ideas of Nash~\cite{Nash54} were first used in the context of the 3D Euler system by De Lellis-Sz\'ekelyhidi in the seminal work~\cite{DLSZ13}. A sequence of works~\cite{DeLellisSzekelyhidi12a,BDLISZ15,Buckmaster15,BDLSZ16,IsettOh16,DSZ17} further built on these ideas, leading to the resolution of the Onsager conjecture by Isett~\cite{Isett18,Isett17}. For  dissipative solutions, the proof of the flexible side of the Onsager conjecture was given by B.-De Lellis-Sz\'ekelyhidi-V.~\cite{BDLSV19}  (see~\cite{DLSZ17,BV19} for recent reviews). Nash-style convex integration schemes in H\"older spaces were also applied to other classical hydrodynamic models~\cite{IsettVicol15,BSV16,CDLDR17,Novack18}. The last two authors' work~\cite{BV} introduced intermittent building blocks in a $L^2$-based convex integration scheme in order to construct weak solutions of the 3D Navier-Stokes equations (3D NSE) with prescribed kinetic energy. These ideas were further developed in~\cite{BCV18} to construct intermittent weak solutions of 3D NSE with partial regularity in time, in~\cite{LuoTiti18,BCV18} for the hyperdissipative problem, in~\cite{Luo18,CheskidovLuo19} for the stationary problem, and in~\cite{Dai18} to treat the Hall-MHD system. We note that Dai's~\cite{Dai18} non-uniqueness result fundamentally relies on the presence of the Hall term $\curl ( \curl B \times B)$ which is of highest order and is not present in the ideal MHD system. We refer to the review papers~\cite{DLSZ17,DLSZ19,BV19} for further references.

The main difficulties in proving Theorem~\ref{thm:main} arise from the specific geometric structure of 3D MHD  which we describe next, along with the main new ideas used to overcome them. 
First, the intermittent constructions developed in the context of 3D NSE~\cite{BV,BCV18,CheskidovLuo19}, more specifically the building blocks of these constructions (intermittent Beltrami flows, intermittent jets, respectively viscous eddies), are not applicable to the ideal MHD system. Informally speaking, for 3D NSE one requires building blocks with {\em more than 2D intermittency}, whereas the geometry of the nonlinear terms of 3D MHD system requires the building blocks' direction of oscillation to be orthogonal to two direction vectors, only permitting  the usage of {\em 1D intermittency}  (co-dimension $2$). In particular, our construction does not work for the 2D MHD system, as expected~\cite{FaracoLindberg18a}. 
Our solution is based on constructing (see Section~\ref{sec:intertmittent:section}) a set of intermittent building blocks adapted to this geometry, which we call {\em intermittent shear velocity flows} and {\em intermittent shear magnetic flows}. Their spatial support is given by a thickened plane spanned by two orthogonal vectors $k_1$ and $k_2$, whereas their only direction of oscillation is given by a vector $k$ which is orthogonal to both $k_1$ and $k_2$. The second fundamental difference is that in 3D NSE intermittency is only used to treat the {\em linear term} $\Delta u$, as an error term. In the case of 3D ideal MHD it turns out that  intermittency is used to treat the {\em nonlinear oscillation terms}.  Due to the two dimensional nature of their support, the interaction of different intermittent shear flows is not small when measured using the usual techniques. At this point intermittency plays a key role: we note that the product of two rationally-skew-oriented 1D intermittent building blocks is {\em more intermittent than each one of them}: it has 2D intermittency because the intersection of two thickened (nonparallel) planes is given by a thickened line, which has 2D smallness.
 
We remark that a similar method to the one outlined here, combined with suitable localization arguments, should be able to yield the existence of weak solutions to ideal 3D MHD which have compact support in physical space and which do not conserve magnetic helicity (see~\cite{Gavrilov19,CLV19} for the construction of smooth and of rough solutions to {\em steady} ideal MHD with compact support). Such a construction would permit the treatment of non-simply-connected domains, an important geometry in plasma physics (e.g.~tokamaks). 

We also note that the construction given in this paper describes an algorithm with very explicit steps. Moreover, as opposed to Euler convex integration schemes, one does not need to numerically solve a large number of transport equations, which is computationally costly. It would be very interesting to implement the  construction given below on a computer, and to visualize the emerging intermittent MHD structures. 

\section{Outline of the paper}

The proof of Theorem \ref{thm:main} relies  on constructing solutions $(u_q, B_q, \mathring{R}_q^u, \mathring{R}_q^B)$ for every integer $q \geq 0$ to the following relaxation of \eqref{eq:ideal:mhd}: 
\begin{subequations}
\label{relax mhd}
\begin{align}
&\partial_t u_q + \div( u_q \otimes u_q - B_q \otimes B_q)  + \nabla p_q = \div \mathring{R}_q^u 
\label{relax mhd 1}\\
&\partial_t B_q + \div( u_q \otimes B_q - B_q \otimes u_q)  = \div \mathring{R}_q^B
\label{relax mhd 2}\\
& \div u_q = \div B_q =  0
\label{relax mhd div free}
\end{align}
\end{subequations}
where $\mathring{R}_q^u$ is a symmetric traceless $3 \times 3$ matrix which we  call the {\em Reynolds stress} and $\mathring{R}_q^B$ is a skew-symmetric $3 \times 3$ matrix which we call the {\em magnetic stress}. We recover the pressure $p_q$ by solving the equation $\Delta p_q = \div \div (-u_q \otimes u_q +B_q \otimes B_q +\mathring{R}_q^u)$ with $\int_{\mathbb{T}^3} p_q dx = 0$. We construct solutions to \eqref{relax mhd} such that the Reynolds and magnetic stresses go to zero in a particular way as $q\to \infty$, so that in the limit we obtain a weak solution of \eqref{eq:ideal:mhd}. 

In order to quantify the convergence of the stresses we introduce a frequency parameter $\lambda_q$ and an amplitude parameter $\delta_q$ defined as follows: 
\begin{equation}
\label{frequency and amplitude parameters}
    \lambda_q = a^{(b^q)} \qquad \mbox{and}  \qquad \delta_q = \lambda_q^{-2\beta} 
\end{equation}
where $\beta > 0$ is  a (very small) regularity parameter  and  $a, b \in \mathbb{N}$ are both large.
By induction, we will assume the following bounds on the solution of \eqref{relax mhd} at level $q$:
\begin{align}
\label{mag inductive assumptions}
\norm{B_q}_{L^2} &\leq 1 - \delta_{q}^{\frac{1}{2}}, \qquad \norm{B_q}_{C_{x,t}^1} \leq \lambda_q^2, \qquad  \norm{ \mathring{R}_q^B }_{L^1} \leq c_B \delta_{q+1} , \\ 
\label{vel inductive assumptions}
\norm{u_q}_{L^2} &\leq 1 - \delta_{q}^{\frac{1}{2}},  \qquad \norm{u_q}_{C_{x,t}^1}  \leq \lambda_{q}^2, \qquad \norm{\mathring{R}_q^u}_{L^1} \leq c_u \delta_{q+1} \,.
\end{align}
The constants $c_u$ and $c_B$ are universal: $c_u$ only depends on fixed geometric quantities, and $c_B$ depends on $c_u$ and other geometric quantities. We can assume that $c_u, c_B \leq 1$. 
We note that, unless otherwise stated, $\norm{f}_{L^p}$ will be used as shorthand for $\norm{ f}_{L_t^{\infty} ((-T, T);L_x^{p}(\mathbb{T}^3))}$. Moreover, we write $\norm{f}_{C^1_{x,t}}$ to denote $\norm{f}_{L^\infty} + \norm{\nabla f}_{L^\infty} + \norm{\partial_t f}_{L^\infty}$.

\begin{proposition}[\bf Main Iteration]
\label{prop:main:iteration}
There exist constants $\beta  \in (0,1)$ and  $a_0 = a_0( \beta, c_B, c_u)$ such that for any natural number $a \geq a_0$ there exist functions $(u_{q+1}, \mathring{R}_{q+1}^{u}, B_{q+1}, \mathring{R}_{q+1}^B )$ which solve \eqref{relax mhd} and satisfy \eqref{mag inductive assumptions} and \eqref{vel inductive assumptions} at level $q+1$. Furthermore, they satisfy 
\begin{equation}
\label{main inductive assumption}
\norm{u_{q+1} - u_q }_{L^2} \leq \delta_{q+1}^{\frac{1}{2}} \qquad \text{  and  } \qquad \norm{B_{q+1} - B_q }_{L^2} \leq \delta_{q+1}^{\frac{1}{2}} \,. 
\end{equation}
\end{proposition}

Sections~\ref{mollification}--\ref{stress} contain the proof of Proposition~\ref{prop:main:iteration}, while the proof of Theorem~\ref{thm:main} is given in Section~\ref{final proof}. 

\section{Mollification}
\label{mollification} 
It is convenient to mollify the velocity and the magnetic field to avoid the loss of derivatives problem. 
Let $\phi_{\epsilon}$ be a family of standard Friedrichs mollifiers on $\mathbb{R}^3$ and let $\varphi_{\epsilon}$ be a family of standard Friedrichs mollifiers on $\mathbb{R}$. Define a mollification of $u_q, B_q, \mathring{R}_q^u, \text{ and, } \mathring{R}_q^B$ in space and time at length scale $\ell$ by
\begin{align*}
u_{\ell} := (u_q *_x \phi_{\ell}) *_t \varphi_{\ell} &\qquad \text{ and }  \qquad B_{\ell} := (B_q *_x \phi_{\ell}) *_t \varphi_{\ell}\\
\mathring{R}_{\ell}^{u} := (\mathring{R}_q^u *_x \phi_{\ell}) *_t \varphi_{\ell} &\qquad \text{ and } \qquad 
\mathring{R}_{\ell}^{B} := (\mathring{R}_q^B *_x \phi_{\ell}) *_t \varphi_{\ell} \, .
\end{align*}
Using \eqref{relax mhd 1} and \eqref{relax mhd 2}, $(u_{\ell},  \mathring{R}_{\ell}^u)$ and $(B_{\ell}, \mathring{R}_{\ell}^B)$ satisfy
\begin{subequations}
\label{mollified relaxed imhd}
\begin{align}
&\partial_t u_{\ell} + \div ( u_{\ell} \otimes u_{\ell}  - B_{\ell} \otimes B_{\ell} ) + \nabla p_{\ell} = \div (\mathring{R}_{\ell}^u + \mathring{R}_{comm}^u)
\label{mollified velo eq}\\
&\partial_t B_{\ell} + \div ( u_{\ell} \otimes B_{\ell}  - B_{\ell} \otimes u_{\ell} )  = \div (\mathring{R}_{\ell}^B + \mathring{R}_{comm}^B)
\label{mollified mag eq}\\
&\div u_{\ell} =  \div B_{\ell} = 0
\end{align}
\end{subequations}
where the traceless symmetric commutator stress $\mathring{R}_{comm}^u$ and the skew-symmetric commutator stress $\mathring{R}_{comm}^B$ are  given by 
\begin{align*}
\mathring{R}_{comm}^u &= (u_{\ell}\mathring{\otimes} u_{\ell}) - (B_{\ell}\mathring{\otimes} B_{\ell}) - (( u_q \mathring{\otimes} u_q - B_q \mathring{\otimes} B_q )*_x \phi_{\ell} ) *_t \varphi_{\ell} \, , \\
\mathring{R}_{comm}^B &= u_{\ell} \otimes B_{\ell} - B_{\ell} \otimes u_{\ell} - ((u_q \otimes B_q - B_q \otimes u_q) *_x \phi_{\ell} ) *_t \varphi_{\ell} \, ,
\end{align*}
and $p_{\ell}$ is defined as 
\begin{equation*}
    p_{\ell} = (p_q *_x \phi_{\ell}) *_t \varphi_{\ell} - |u_{\ell}|^2 + |B_{\ell}|^2 + (|u_q|^2 - |B_{q}|^2)*_x \phi_{\ell} ) *_t \varphi_{\ell} \, .
\end{equation*}

Using standard mollification estimates and \eqref{mag inductive assumptions}--\eqref{vel inductive assumptions} we have the following estimates for $\mathring{R}_{\ell}^B$ and $\mathring{R}_{\ell}^u$:
\begin{equation}
\label{est:mollified:stress}
\norm{\nabla^M \mathring{R}_{\ell}^u}_{L^1} + \norm{ \nabla^M \mathring{R}_{\ell}^B}_{L^1}  \lesssim \ell^{-M}\delta_{q+1} \,.
\end{equation}
For $\mathring{R}_{comm}^B$ we use the double commutator estimate from \cite{ConstantinETiti94} and the inductive estimates \eqref{mag inductive assumptions}--\eqref{vel inductive assumptions} to conclude
\begin{equation}
\label{mag commutator estimate}
\norm{\mathring{R}_{comm}^B}_{L^1} \lesssim \norm{\mathring{R}_{comm}^B}_{C^0} \lesssim \ell^2 \norm{B_q}_{C_{x,t}^1}\norm{ u_q}_{C_{x,t}^1}  \lesssim  \ell^2 \lambda_q^{4} \,. 
\end{equation}
Since $u_q$ and $B_q$ satisfy the same inductive estimates, we have the same bound from \eqref{mag commutator estimate}:
\begin{equation}
\label{vel commutator estimate}
\norm{\mathring{R}_{comm}^u}_{L^1} \lesssim \ell^2 \lambda_{q}^4\, .
\end{equation}
We will choose the mollification length scale so that both \eqref{mag commutator estimate} and \eqref{vel commutator estimate} are  less than $\delta_{q+2}$: using  \eqref{frequency and amplitude parameters} this  implies that $\ell$ must satisfy  
\begin{equation}
\label{mollification upperbound}
    \ell \ll \lambda_{q+1}^{-\frac{2}{b} - \beta b}.
\end{equation}
If we define $\ell$ as 
\begin{equation*}
    \ell := \lambda_{q+1}^{-\eta} 
\end{equation*}
then \eqref{mollification upperbound} translates into  $\eta >  \frac{2}{b} + \beta b$.

\begin{remark}
The implicit constants appearing in \eqref{mag commutator estimate} and \eqref{vel commutator estimate}, as well as later inequalities in this paper, will depend on the mollifiers, $N_{\Lambda}$ (see Remark~\ref{Universal constant}), $\Phi$ (see Section~\ref{sec:intertmittent:section}), and various other geometric quantities. In particular, none of the implicit constants will depend on $q$. By taking $a$ to be sufficiently large we will be able to use a small power of $\lambda_{q+1}$ to absorb the implicit constants and have bonafide inequalities.    
\end{remark}

\section{Linear Algebra}
As with previous convex integration schemes, we construct perturbations to add to the velocity and magnetic fields to reduce the size of the stresses. The following two lemmas are an important part of designing the perturbations so that this cancellation of the previous stress occurs. The proofs are  given in Appendix~\ref{Proof of Geometric lemmas}.

\begin{lemma}[\bf First Geometric Lemma]
\label{geometric lem 1}
There exists a set $\Lambda_B \subset S^2 \cap \mathbb{Q}^3$ that consists of vectors $k$ with associated orthonormal bases $(k, k_1, k_2)$,  $\varepsilon_B > 0$, and smooth positive functions $\gamma_{(k)}: B_{\varepsilon_B}(0) \to \mathbb{R}$, where $B_{\varepsilon_B}(0)$ is the ball of radius $\varepsilon_B$ centered at 0 in the space of $3 \times 3$ skew-symmetric matrices, such that for  $A \in B_{\varepsilon_B}(0)$ we have the following identity: 
\begin{equation}
\label{antisym}
A = \sum_{k \in \Lambda_B} \gamma_{(k)}^2(A) (k_1 \otimes k_2 - k_2 \otimes k_1)  \,.
\end{equation}
\end{lemma}

\begin{lemma}[\bf Second Geometric Lemma]
\label{geometric lem 2}
There exists a set $\Lambda_u \subset S^2 \cap \mathbb{Q}^3$ that consists of vectors $k$ with associated orthonormal bases $(k, k_1, k_2)$,  $\varepsilon_u > 0$, and smooth positive functions $\gamma_{(k)}: B_{\varepsilon_u}(\Id) \to \mathbb{R}$, where $B_{\varepsilon_u}(\Id)$ is the ball of radius $\varepsilon_u$ centered at the identity in the space of $3 \times 3$ symmetric matrices,  such that for  $S \in B_{\varepsilon_u}(\Id)$ we have the following identity:
\begin{equation}
\label{sym}
S = \sum_{k \in \Lambda_u} \gamma_{(k)}^2(S) k_1 \otimes k_1   \,.
\end{equation} 
Furthermore, we may choose $\Lambda_u$ such that $\Lambda_B \cap \Lambda_u = \emptyset$. 
\end{lemma}

\begin{remark}
\label{Universal constant}
By our choice of $\Lambda_B$ and $\Lambda_u$ and the associated orthonormal bases, there exists $N_{\Lambda} \in \mathbb{N}$ with
\begin{equation*}
\{ N_{\Lambda} k,N_{\Lambda}k_1 , N_{\Lambda}k_2 \} \subset N_{\Lambda} \mathbb{S}^2 \cap \mathbb{Z}^3.
 \end{equation*} 
For instance, $N_{\Lambda} = 65$ suffices.
\end{remark}

\begin{remark}
Let $M_*$ be a  geometric constant such that
\begin{align}
\sum_{k \in \Lambda_{u}} \norm{\gamma_{(k)}}_{C^1(B_{\varepsilon_u}(\Id))} + \sum_{k \in \Lambda_{B}} \norm{\gamma_{(k)}}_{{C^1(B_{\varepsilon_B}}(0))} \leq M_* \,.
 \label{M bound}
\end{align}
 This parameter  is universal. We will need this parameter later when estimating the size of the perturbations, see \eqref{vel amp L2 estimate} and \eqref{mag amp L2}. 
\end{remark}

\section{Constructing the Perturbation: Intermittent Shear flows}
\label{sec:intertmittent:section}
Let $\Phi : \mathbb{R} \to \mathbb{R}$ be a smooth cutoff function supported on the interval $[-1,1]$. Assume it is normalized in such a way that $\phi := - \frac{d^2}{dx^2}\Phi$ satisfies 
\begin{equation*}
\int_{\mathbb{R}} \phi^2(x)dx = 2\pi. 
\end{equation*}
For a small parameter $r$, define the rescaled functions  
\begin{equation*}
\phi_r(x) := \frac{1}{r^{\frac{1}{2}}}\phi\left(\frac{x}{r}\right), \qquad \mbox{and} \qquad 
\Phi_r(x):=  \frac{1}{r^{\frac{1}{2}}} \Phi\left(\frac{x}{r}\right), 
\end{equation*}
which implies the relation $\phi_r = -r^2 \frac{d^2}{dx^2} \Phi_r$.
We periodize $\phi_r$ and $\Phi_r$ so that we can view the resulting functions (which we will also denote as $\phi_r$ and $\Phi_r$) as functions defined on $\mathbb{R}/2\pi \mathbb{Z} =  \mathbb{T}$. For a large parameter $\lambda$  such that $\lambda^{-1} \ll r$ and $r \lambda \in \mathbb{N}$ the \textit{intermittent shear velocity flow} is defined as  
\begin{equation*}
W_{(k)} :=  \phi_r( \lambda r N_{\Lambda}k\cdot x)k_1  \qquad \text{ for } \qquad k \in \Lambda_u \cup \Lambda_B \, ,
\end{equation*}
and the \textit{intermittent shear magnetic flow} is defined as 
\begin{equation*}
D_{(k)} :=  \phi_r( \lambda r N_{\Lambda}k\cdot x)k_2 \qquad \text{ for } \qquad k \in \Lambda_B \, ,
\end{equation*}
where the notation $(k)$ at the subindex is shorthand for a dependence on $k, \lambda \text{ and }$other parameters. The fields 
$W_{(k)}$ and $D_{(k)}$ are $(\mathbb{T}/(r \lambda))^3-$ periodic, have zero mean, and are divergence free. We introduce the shorthand notation 
\begin{equation*}
\phi_{(k)}(x) := \phi_r( \lambda r N_{\Lambda}k\cdot x), \qquad \Phi_{(k)}(x) := \Phi_r( \lambda r N_{\Lambda}k\cdot x)
\end{equation*}
which allows us to write the intermittent fields more concisely as
\begin{equation*}
W_{(k)} = \phi_{(k)} k_1,  \qquad D_{(k)} = \phi_{(k)} k_2 \,.
\end{equation*}
Note that by the choice of normalization for $\phi$,   we have
\begin{equation}
\label{normalization}
\left \langle \phi_{(k)}^2\right\rangle = \fint_{\mathbb{T}^3} \phi_{(k)}^2(x) dx = 1 \, .
\end{equation}
This sets the zeroth Fourier coefficient for $\phi_{(k)}^2$ to equal $1$ and implies that $\norm{W_{(k)}}_{L^2}^2 = \norm{D_{(k)}}_{L^2}^2 = 8\pi^3$.

\subsection{Estimates for $W_{(k)}$ and $D_{(k)}$}

\begin{lemma}
\label{linear estimate lemma}
For $p \in [1,\infty]$ and $M \in \mathbb{N}$ we have the following estimates for $\phi_{(k)}$ and $\Phi_{(k)}$:
 \begin{equation}
\label{linear intermittent  estimates}
 \norm{\nabla^M \Phi_{(k)}}_{L^p}  + \norm{\nabla^M \phi_{(k)}}_{L^p} \lesssim  \lambda^M r^{\frac{1}{p} - \frac{1}{2}}.
\end{equation} 
Furthermore, we have the following estimate for the size of the support of $\phi_{(k)}$:
\begin{equation}
\label{est:support:phi}
\abs{\supp(\phi_{(k)})} \lesssim r 
\end{equation}
where $| \cdot |$ denotes Lebesgue measure and the implicit constant only depends on the wavevector sets and fixed geometric quantities. 
\end{lemma}

\begin{proof}[Proof of Lemma~\ref{linear estimate lemma}]

First, we estimate the $L^{\infty}$ norm.  Let $\alpha$ be a multiindex such that $|\alpha | = M$. Then, 
\begin{align*}
\partial_x^{\alpha}\phi_{(k)}(x) = \partial_x^{\alpha}(\phi_r(N_{\Lambda} \lambda r k \cdot x   )    ) 
&= k^{\alpha} (N_{\Lambda} r \lambda )^{M}   \frac{d^{M}}{dx^M}\phi_r(N_{\Lambda} \lambda r k \cdot x   )
\end{align*}
where $k^{\alpha} = \prod_{i=1}^3 k_i^{\alpha_i}$. 
Using the definition of $\phi_r$ we have that 
\begin{equation*}
 \frac{d^{M}}{dx^M} \phi_r(N_{\Lambda} r \lambda k \cdot x) = \frac{1}{r^{\frac{1}{2} + M} }  \frac{d^{M}}{dx^M} \phi( N_{\Lambda} \lambda k \cdot x) \,.
\end{equation*}
Since $\phi$ is a smooth compactly supported function this implies that
\begin{equation}
\label{est:max:phi}
\norm{ \nabla^M\phi_{(k)}}_{L^{\infty}} \lesssim \lambda^M r^{-\frac{1}{2}} \,.
\end{equation}

Next, we estimate the $L^1$ norm. To do this, we first obtain a bound on the size of the support of $\phi_{(k)}$, as claimed in \eqref{est:support:phi}. Recall that $\phi_{(k)}$ is $(\mathbb{T}/( \lambda r))^3$-periodic. Therefore, $\phi_{(k)}$ on $\mathbb{T}^3$ can be thought of as being made of $(\lambda r)^3$ copies of $\phi_{(k)}$ defined on cubes of side length $\frac{2 \pi}{\lambda r}$.  Thus, it suffices to obtain an estimate on cubes with side length $\frac{2\pi}{\lambda r}$ and then multiply the resulting estimates by $(\lambda r)^3$. Due to the periodicity of $\phi_{(k)}$, in one of these cubes the support of $\phi_{(k)}$ consists of parallel planes with thickness $\sim \lambda^{-1}$. The minimum distance between the planes is bounded below by $s \frac{2 \pi}{\lambda r}$ where $s \in (0,1)$ depends only on the wavevector sets (specifically, $s$ is the minimum distance from the planes determined by $k\cdot x = 0 $ to a point in $(2 \pi \mathbb{Z})^3$; by the rationality of the entries of $k$ and since there are only a finite number of wavevectors this number is finite). Since the side length of the cubes is $\frac{2 \pi}{\lambda r}$, the the maximum number of thickened planes that could compose the support of $\phi_{(k)}$ is bounded by $2s^{-1}$. Therefore, over the small cube we have a support bound given by $ |\supp(\phi_{(k)})| \leq C_{\Lambda_u, \Lambda_B} (\lambda r)^{-2} \lambda^{-1} $ where $C_{\Lambda_u, \Lambda_B}$ is a constant depending on the wavevector sets and other geometric quantities. Multiplying this bound by $(\lambda r)^3$ gives the desired support estimate for whole torus.

The $L^1$ estimate follows from the support bound. Using H\"older's inequality, \eqref{est:max:phi}, and \eqref{est:support:phi} we have
\begin{equation*}
    \norm{\nabla^M \phi_{(k)} }_{L^1} \leq \abs{\supp(\nabla^M \phi_{(k)} )} \norm{\nabla^M \phi_{(k)}}_{L^{\infty}} \lesssim \abs{\supp(\phi_{(k)} )}  \lambda^M r^{-\frac{1}{2}} \lesssim \lambda^M r^{\frac{1}{2}} \,.
\end{equation*}
Interpolating between the $L^1$ and $L^{\infty}$ yields the desired estimate for all $p\in (1,\infty)$.  
Repeating the same analysis for $\Phi_{(k)}$ gives the desired conclusion. 
\end{proof}

\begin{lemma}[\bf Product estimate]
\label{product estimate lemma}
For $p \in [1, \infty]$, $M \in \mathbb{N}$, and $k \neq k'$ we have the following  estimate 
\begin{equation}
\label{prod_est_highfreq}
\norm{\nabla^M( \phi_{(k)} \phi_{(k')}) }_{L^p(\mathbb{T}^3)} \lesssim \lambda^M r^{\frac{2}{p}  -1}.
\end{equation}
Furthermore, we have the following estimate for the size of the support of $\phi_{(k)} \phi_{(k')}$: 
\begin{equation}
\label{product support}
\abs{\supp(\phi_{(k)}\phi_{(k')})}\lesssim r^2
\end{equation}
where the implicit constant only depends on the wavevector sets and fixed geometric quantities. 
\end{lemma}

\begin{proof}[Proof of Lemma~\ref{product estimate lemma}]

Proceeding as before, we first estimate the $L^{\infty}$ norm. Using \eqref{linear intermittent  estimates}  with $p = \infty$ yields
\begin{align}
\label{est:prod:phi:max}
    \norm{\nabla^M (\phi_{(k)} \phi_{(k')}) }_{L^{\infty}} \lesssim \sum_{j = 0}^M \norm{ \nabla^j \phi_{(k)}}_{L^{\infty}} \norm{\nabla^{M-j}\phi_{(k')} }_{L^{\infty}}  \lesssim \lambda^M r^{-1} \,.
\end{align}
We now obtain a bound on the support of the function $\phi_{(k)}\phi_{(k')}$ for $k \neq k'$. As in the proof of Lemma \ref{linear estimate lemma} it suffices to obtain an estimate on cubes with side length $\frac{2\pi}{\lambda r}$ and then multiply the resulting estimates by $(\lambda r)^3$. Since the support of $\phi_{(k)}$ consists of parallel planes with thickness $\sim \lambda^{-1}$, the support of $\phi_{(k)} \phi_{(k')}$ will consist of the intersection of these thickened planes, which are thickened lines with cross-sectional area $\sim \frac{\lambda^{-2}}{\sin(\theta)}$ where $\theta$ is the angle between $k$ and $k'$. Since there are only a finite number of wavevectors, there is a minimal separation angle $\theta$. Therefore the cross-sectional area for an individual cylinder is bounded by $C_{\Lambda_u, \Lambda_B}\lambda^{-2}$ where $C_{\Lambda_u, \Lambda_B}$ is some constant depending on the wavevector sets and other geometric quantities. To estimate the total number of intersections of the planes in a given cube, we note that since the total number of thickened planes in the support of  $\phi_{(k)}$ in a small cube is bounded by $2s^{-1}$ the number of intersection points for two distinct planes is bounded by $4s^{-2}$. Finally, the length of such an intersection is bounded by the main diagonal of the cube, therefore it is bounded by $2 \lambda r$.  Combining all of this, we conclude that, over an individual cube with side length $ \frac{2 \pi}{\lambda r}$, the measure of the support of $\phi_{(k)} \phi_{(k')}$ is bounded by $C_{\Lambda_u, \Lambda_B}\lambda^{-2}(\lambda r )^{-1}$. Multiplying by the total number of cubes $(\lambda r)^{3}$ gives the bound $ C_{\Lambda_u, \Lambda_B} r^2$ in \eqref{product support}. 

We now proceed with the $L^1$ estimate using H\"older's inequality, \eqref{est:prod:phi:max}, and \eqref{product support}: 
\begin{align*}
    \norm{ \nabla^M (\phi_{(k)}\phi_{(k')} ) }_{L^1} \leq |\supp(\nabla^M (\phi_{(k)}\phi_{(k')} ))| \norm{\nabla^M (\phi_{(k)}\phi_{(k')} )}_{L^{\infty}} \lesssim |\supp(\phi_{(k)}\phi_{(k')} )| \lambda^M r^{-1}  \lesssim \lambda^M r \,.
\end{align*}
By interpolation between the $L^1$ and $L^{\infty}$ norms we obtain the desired result. 
\end{proof}

We will now fix the values of the parameters $r$ and $\lambda$. We set 
\[ 
\lambda := \lambda_{q+1} \qquad \mbox{and} \qquad r := \lambda_{q+1}^{-\frac{3}{4}}.
\] 
The requirement that $r\lambda \in \mathbb{N} = \lambda_{q+1}^{-\frac{3}{4}}$ implies that $b$ from \eqref{frequency and amplitude parameters} should be divisible by 4. 

\begin{remark}
Now that we have defined all the fundamental parameters, we can specify values that allow the proof of Proposition~\ref{prop:main:iteration} to close. If we let $\beta = 10^{-9}$ then  $b = 10^4$ and $\eta = 10^{-3}$ are allowable choices.   
\end{remark}

\subsection{The Perturbation}
\subsubsection{Amplitudes}
To apply the geometric lemmas we need pointwise control over the size of the stresses. However, the stresses are not necessarily spatially homogeneous, so we need to divide them by suitable functions to ensure that they are pointwise small, as well as small in $L^1$.  To achieve this, we follow \cite{LuoTiti18}.  Let $\chi: [0, \infty) \to \mathbb{R}$ be a smooth function satisfying 
\begin{equation*}
\chi (z) = 
 \begin{cases} 
      1 & 0 \leq z\leq 1 \\
      z & z \geq 2  
   \end{cases}
\end{equation*}
with $z \leq 2\chi(z) \leq 4z $ for $z \in (1,2)$. 

Next, we define
\begin{equation*}
\rho_B(x,t) := 2\delta_{q+ 1} \varepsilon_B^{-1} c_{B}\chi\left( (c_B\delta_{q+1})^{-1} |\mathring{R}_{\ell}^B(x,t) | \right) \,
\end{equation*}
where $\eps_{B}$ is as in Lemma \ref{geometric lem 1}. The key properties of $\rho_B$ are that pointwise we have
\begin{equation*}
\left|  \frac{\mathring{R}_{\ell}^B(x,t)}{\rho_B(x,t)} \right| = \left| \frac{\mathring{R}_{\ell}^B(x,t)}{2\delta_{q+1} \varepsilon_B^{-1} c_{B}\chi\left( (c_B\delta_{q+1})^{-1} |\mathring{R}_{\ell}^B(x,t) | \right)} \right| \leq \varepsilon_B
\end{equation*}
and that for all $p \in [1, \infty)$ the bound
\begin{equation}
\label{mag divisor lp estimate}
\norm{ \rho_B }_{L^p} \leq 8  \varepsilon_B^{-1}\left( (c_B (8\pi^3)^{\frac{1}{p}})\delta_{q+1} + \norm{\mathring{R}_{\ell}^{B} }_{L^p} \right) 
\end{equation}
holds. By using standard H\"older estimates (see, for example, \cite[Appendix C]{BDLISZ15}), \eqref{est:mollified:stress}, the ordering $\ell \leq \delta_{q+1}$, and the gain of integrability for mollified functions we have
\begin{subequations}
\begin{align}
\label{Derivative estimates mag divisor}
\norm{ \rho_B }_{C_{x,t}^0} \lesssim  \ell^{-3} \qquad  \mbox{ and } \qquad \norm{ \rho_B }_{C_{x,t}^j}  \lesssim \ell^{-4j}\\
\norm{ \rho_B^{\frac{1}{2}} }_{C_{x,t}^0 } \lesssim \ell^{-2} \qquad  \mbox{ and } \qquad \norm{  \rho_B^{\frac{1}{2}} }_{C_{x,t}^j} \lesssim \ell^{-5j}  \, .
\label{Derivative estimates mag sqrt divisor}
\end{align}
\end{subequations}
for $j\geq 1$.

We then define the {\em magnetic amplitude functions} 
\begin{align}
\label{def:mag:amp}
a_{(k)}:= a_{k, B}(x,t) =  \rho_B^{\frac{1}{2} }  \gamma_{(k)} \left(\frac{-\mathring{R}_{\ell}^B}{\rho_B}   \right), \qquad \mbox{ for } \qquad k \in \Lambda_B \,.
\end{align}
By \eqref{mag divisor lp estimate},  \eqref{mag inductive assumptions}, the fact that mollifiers have mass $1$,  and by choosing $c_B$ sufficiently small, we have 
\begin{align}
\label{mag amp L2}
\norm{a_{k,B} }_{L^2} &\leq \norm{ \rho_B }_{L^1}^{\frac{1}{2}} \norm{ \gamma_{(k)} }_{C^0(B_{\varepsilon_B}(0))} \notag\\
                  &\leq M_* (8 \varepsilon_B^{-1})^{\frac{1}{2}}( c_B 8\pi^3\delta_{q+1} + \norm{\mathring{R}_{\ell}^{B} }_{L^1})^{\frac{1}{2}} \notag \\
                  &\leq M_* [  8 \varepsilon_B^{-1}c_B \delta_{q+1} (8\pi^3 + 1)    ]^{\frac{1}{2}} \notag \\
                  &\leq  \min\left[ \left(\frac{c_u}{|\Lambda_B|}\right) ^{\frac{1}{2}}, \frac{1}{3|\Lambda_B| C_* (8 \pi^3)^{\frac{1}{2}}} \right] \delta_{q+1}^{\frac{1}{2}}\, .
\end{align}
where $C_*$ is defined in Lemma \ref{Decorrelation}.
The reason for the strange prefactor in front of the $\delta_{q+1}^{\frac{1}{2}}$ is because the magnetic amplitudes will be used to define two different objects which need to satisfy different sets of bounds (for details, see the discussion preceding \eqref{G estimates} and \eqref{Lp decorr mag} below).
Using~\eqref{Derivative estimates mag sqrt divisor} we arrive at 
\begin{equation}
\label{mag amp estimates}
 \norm{ a_{(k)} }_{C_{x,t}^j} \lesssim \ell^{-5j -2 } \,.
\end{equation}
for $j\geq 0$.

The motivation for definition \eqref{def:mag:amp} is as follows: by \eqref{normalization} we have 
\begin{align*}
\phi_{(k)}^2 (k_1 \otimes k_2 - k_2 \otimes k_1) &= \langle \phi_{(k)}^2 \rangle (k_1 \otimes k_2 - k_2 \otimes k_1) + \mathbb{P}_{\neq 0}(\phi_{(k)}^2)(k_1 \otimes k_2 - k_2 \otimes k_1)   \\
                                   &= k_1 \otimes k_2 - k_2 \otimes k_1 + \mathbb{P}_{\neq 0}(\phi_{(k)}^2)(k_1 \otimes k_2 - k_2 \otimes k_1) 
\end{align*}
where $\langle \cdot \rangle$ denotes spatial average over $\mathbb{T}^3$ and $\mathbb{P}_{\neq 0}$ denotes projection onto nonzero Fourier modes.
Multiplying through by $a_{(k)}^2$, summing over $\Lambda_B$, and using Geometric Lemma 1 gives
\begin{equation}
\label{eq:mag:cancel:eqn}
 \sum_{ k \in  \Lambda_B} a_{(k)}^2 \phi_{(k)}^2(k_1 \otimes k_2 - k_2 \otimes k_1)  
= -\mathring{R}_{\ell}^B + \sum_{ k \in \Lambda_B} a_{(k)}^2 \mathbb{P}_{\neq 0}(\phi_{(k)}^2)(k_1 \otimes k_2 - k_2 \otimes k_1)  \, .
\end{equation}

Before we give the definition of the velocity amplitude functions we note that we need to account for two key differences with the magnetic amplitudes: Geometric Lemma 2 allows us to cancel matrices in a neighborhood of the identity as opposed to the origin. In order to cancel both stresses, the velocity perturbation will need to have wavevectors from both $\Lambda_u$ and $\Lambda_B$  (see \eqref{principal velo ptb}). To address this second issue we define  
\begin{equation}
\label{def:G}
\mathring{G}^B: = \sum_{k \in \Lambda_B}a_{(k)}^2 (k_1 \otimes k_1 - k_2 \otimes k_2).
\end{equation}
Note that since $\mathring{G}^B$ only depends on $a_{(k)}$, we have that $\mathring{G}^B$ is a function of $\mathring{R}_{\ell}^B$.
By using that $a_{(k)}^2 = \rho_B \gamma_{(k)}^2(-\frac{R_{\ell}^B}{\rho_B})$, \eqref{est:mollified:stress},   \eqref{Derivative estimates mag divisor},  and  \eqref{mag amp L2}, for $j\geq 0$  we have
\begin{align}
\label{G estimates}
\norm{\mathring{G}^B}_{C_{x,t}^0} \lesssim \ell^{-3} , \qquad \mbox{ and } \qquad \norm{ \mathring{G}^B}_{L^1} &\leq 2c_u\delta_{q+1} \, .
\end{align}

Next, define $\rho_u$  and the associated {\em velocity amplitudes} as
\begin{align}
\rho_u &:= 2 \varepsilon_u^{-1}c_u\delta_{q+ 1}\chi\left( (c_u\delta_{q+1})^{-1} |\mathring{R}_{\ell}^u(x,t) + \mathring{G}^B| \right)
\, , \notag \\
a_{(k)} := a_{k, u}(x,t) &= \rho_u^{\frac{1}{2}}\gamma_{(k)} \left(\Id - \frac{\mathring{R}_{\ell}^u + \mathring{G}^B}{\rho_u} \right) \, ,
\qquad \mbox{ for } \qquad k \in \Lambda_u \,.
\label{def:vel:amp}
\end{align}
Comparing \eqref{def:mag:amp} and \eqref{def:vel:amp} we notice that the definitions of $a_{(k)}$ for $k\in \Lambda_B$, respectively for $k \in \Lambda_u$, differ slightly. Throughout the paper we abuse this notation and write $a_{(k)} = a_{k,B}$ for $k\in\Lambda_B$ and also $a_{(k)}=a_{k,u}$ for $k \in \Lambda_u$.
With these definitions we have the following  properties for $\rho_u$ and $a_{(k)}$:
\begin{equation*}
\left|  \frac{\mathring{R}_{\ell}^u(x,t) + \mathring{G}^B}{\rho_u(x,t)} \right| = \left| \frac{\mathring{R}_{\ell}^u(x,t) + \mathring{G}^B}{2\delta_{q+ 1} \varepsilon_u^{-1} c_{u}\chi\left( (c_u\delta_{q+1})^{-1} |\mathring{R}_{\ell}^u(x,t) + \mathring{G}^B| \right)} \right| \leq \varepsilon_u
\end{equation*}
and we have for all $p \in [1, \infty)$
\begin{equation}
\label{vel divisor lp estimate}
\norm{ \rho_u }_{L^p} \leq 8  \varepsilon_u^{-1}\left( (c_u (8\pi^3)^{\frac{1}{p}})\delta_{q+1} + \norm{\mathring{R}_{\ell}^u(x,t) + \mathring{G}^B}_{L^p} \right) \,.
\end{equation}
Using \eqref{mag amp L2}, \eqref{vel divisor lp estimate}, the fact that mollifiers have mass $1$, \eqref{vel inductive assumptions}, and   choosing $c_u$ sufficiently small we have 
\begin{align}
\label{vel amp L2 estimate}
\norm{a_{(k)}}_{L^2}  \leq \norm{\rho_u}_{L^1}^{\frac{1}{2}}\norm{ \gamma_{(k)} }_{C^0(B_{\varepsilon_u}(\Id))}  & \leq M_*(8 \varepsilon_u^{-1}(c_u 8\pi^3\delta_{q+1} + \norm{\mathring{R}_{\ell}^u }_{L^1}  +  \norm{ \mathring{G}^B}_{L^1})      )^{\frac{1}{2}} \notag \\
&\leq M_* (8 \varepsilon_u^{-1}(c_u 8\pi^3\delta_{q+1} + c_u \delta_{q+1}  + 2c_u \delta_{q+1} )      )^{\frac{1}{2}} \notag \\
&\leq \delta_{q+1}^{\frac{1}{2}}c_u^{\frac{1}{2}}M (8 \varepsilon_u^{-1}( 8\pi^3 + 3) )^{\frac{1}{2}} \notag \\
&\leq \frac{\delta_{q+1}^{\frac{1}{2}}}{ 3 |\Lambda_u| C_* (8 \pi^3)^{\frac{1}{2}}  } \,.
\end{align}
Note that $c_u$ only depends on $M_*$ and $\Lambda_B$ which are fixed at the beginning of the induction. In particular, $c_u$ does not depend on the value of $c_B$ so there is no circular reasoning caused by  $c_B$ depending on $c_u$.
 Using the same techniques used to derive \eqref{mag amp estimates} with \eqref{G estimates} we have for $j \geq 0$
\begin{equation}
\label{vel amp estimates}
 \norm{ a_{(k)} }_{C_{x,t}^j} \lesssim \ell^{-10j -2} \qquad \mbox{for} \qquad k \in \Lambda_u \, .
\end{equation}

Analogous reasoning to that used in \eqref{eq:mag:cancel:eqn} for the coefficients defined for  $k \in \Lambda_u$ gives
\begin{align}
\sum_{k \in \Lambda_u } a_{(k)}^2 \phi_{(k)}^2k_1 \otimes k_1 =  \rho_u \Id - \mathring{R}_{\ell}^u - \mathring{G}^B+ \sum_{k \in \Lambda_u} a_{(k)}^2 \mathbb{P}_{\neq 0}(\phi_{(k)}^2)k_1 \otimes k_1 \,.
\label{eq:vel:cancel:eqn}
\end{align}
Thus, if we  define the the \textit{principal part of the perturbations} $w_{q+1}^p$ and $d_{q+1}^p$  as  
\begin{subequations}
\label{principal ptb}
\begin{align}
w_{q+1}^p &:= \sum_{k \in \Lambda_u } a_{(k)} W_{(k)} + \sum_{k \in \Lambda_B} a_{(k)}W_{(k)}
\label{principal velo ptb}\\
d_{q+1}^p &:= \sum_{k \in \Lambda_B} a_{(k)} D_{(k)} \, ,
\label{principal mag ptb}
\end{align}
\end{subequations}
then in the nonlinear term in the magnetic  equation we can use \eqref{eq:mag:cancel:eqn} to write
\begin{align}
&w_{q+ 1}^p \otimes  d_{q+ 1}^p - d_{q+ 1}^p \otimes w_{q+1}^p + \mathring{R}_{\ell}^B \notag \\
&= \sum_{k \in \Lambda_B} a_{(k)}^2 \phi_{(k)}^2 (k_1 \otimes k_2 - k_2 \otimes k_1) + \mathring{R}_{\ell}^B  + 
 \sum_{k \neq k' \in \Lambda_B} a_{(k)}a_{(k')}\phi_{(k)}\phi_{(k')}(k_1\otimes k_2' - k_2'\otimes k_1) \notag \\
 &\quad + \sum_{k \in \Lambda_u, k' \in \Lambda_B} a_{(k)}a_{(k')}\phi_{(k)}\phi_{(k')}(k_1\otimes k_2' - k_2'\otimes k_1) \notag \\
 &= \sum_{k \in \Lambda_B} a_{(k)}^2 \mathbb{P}_{\neq 0}(\phi_{(k)}^2)(k_1 \otimes k_2 - k_2 \otimes k_1) + \sum_{ k \neq k' \in \Lambda_B} a_{(k)}a_{(k')}\phi_{(k)}\phi_{(k')}(k_1\otimes k_2' - k_2'\otimes k_1) \notag \\
 &\quad + \sum_{k \in \Lambda_u, k' \in \Lambda_B} a_{(k)}a_{(k')}\phi_{(k)}\phi_{(k')}(k_1\otimes k_2' - k_2'\otimes k_1) \, ,
 \label{mag oscillation cancellation calculation}
\end{align}
while for the velocity equation we have that
\begin{align}
&w_{q+1}^p \otimes w_{q+ 1}^p - d_{q+1}^p \otimes d_{q+1}^p + \mathring{R}_{\ell}^u \notag \\
&= \sum_{k, k' \in \Lambda_u} a_{(k)}a_{(k')}\phi_{(k)}\phi_{(k')}k_1\otimes k_1' 
+ \sum_{k, k' \in \Lambda_B} a_{(k)}a_{(k')}\phi_{(k)}\phi_{(k')}( k_1\otimes k_1'-k_2\otimes k_2') + \mathring{R}_{\ell}^u 
\notag \\
 &\quad + \sum_{ k \in \Lambda_u, k' \in \Lambda_B}  a_{(k)}a_{(k')}\phi_{(k)}\phi_{(k')}(k_1\otimes k_1'  + k_1'\otimes k_1)    \notag\\
&=: \OO_1 + \OO_2
 \label{eq:osc:cancellation:1}
\end{align}
where the terms $\OO_1$ and $\OO_2$ are defined by the first, respectively second line of the above.
Using the identity
\[
\sum_{k\in \Lambda_B} a_{(k)}^2 \phi_{(k)}^2 (k_1 \otimes k_1 - k_2 \otimes k_2 ) =  \mathring{G}^B+ \sum_{k\in \Lambda_B} a_{(k)}^2 \mathbb{P}_{\neq 0}(\phi_{(k)}^2) (k_1 \otimes k_1 - k_2 \otimes k_2 ) \, ,
\]
which follows from \eqref{def:G}, and appealing to \eqref{eq:vel:cancel:eqn}, we rewrite the $\OO_1$ term as
\begin{align}
\OO_1 &= \sum_{ k \in \Lambda_u} a_{(k)}^2 \phi_{(k)}^2 k_1 \otimes k_1 
+ \sum_{ k \in \Lambda_B} a_{(k)}^2 \phi_{(k)}^2 (k_1 \otimes k_1- k_2 \otimes k_2 )
+ \mathring{R}_{\ell}^u  \notag\\
 &\quad + \sum_{k \neq k' \in \Lambda_u} a_{(k)}a_{(k')}\phi_{(k)}\phi_{(k')}k_1\otimes k_1' 
 + \sum_{ k \neq  k' \in \Lambda_B} a_{(k)}a_{(k')}\phi_{(k)}\phi_{(k')} ( k_1\otimes k_1'- k_2\otimes k_2')
  \notag \\
 &=  \rho_u \Id - \mathring{R}_{\ell}^u - \mathring{G}^B  + \mathring{R}_{\ell}^u  + \mathring{G}^B + \sum_{ k \in \Lambda_u} a_{(k)}^2 \mathbb{P}_{\neq 0}(\phi_{(k)}^2)k_1 \otimes k_1   + \sum_{k \in \Lambda_B}a_{(k)}^2 \mathbb{P}_{\neq 0}(\phi_{(k)}^2)(k_1 \otimes k_1 - k_2 \otimes k_2) \notag \\
 &\quad + \sum_{k \neq k' \in \Lambda_u} a_{(k)}a_{(k')}\phi_{(k)}\phi_{(k')}k_1\otimes k_1' 
 + \sum_{ k \neq  k' \in \Lambda_B} a_{(k)}a_{(k')}\phi_{(k)}\phi_{(k')} ( k_1\otimes k_1'- k_2\otimes k_2')
  \notag \\
 &= \rho_u \Id +  \sum_{ k \in \Lambda_u} a_{(k)}^2 \mathbb{P}_{\neq 0}(\phi_{(k)}^2)k_1 \otimes k_1  + \sum_{k \in \Lambda_B}a_{(k)}^2 \mathbb{P}_{\neq 0}(\phi_{(k)}^2)(k_1 \otimes k_1 - k_2 \otimes k_2) \notag\\
&\quad + \sum_{k \neq k' \in \Lambda_u} a_{(k)}a_{(k')}\phi_{(k)}\phi_{(k')}k_1\otimes k_1' 
 + \sum_{ k \neq  k' \in \Lambda_B} a_{(k)}a_{(k')}\phi_{(k)}\phi_{(k')} ( k_1\otimes k_1'- k_2\otimes k_2')
  \,.
 \label{eq:osc:cancellation:2}
\end{align}
Therefore, combining \eqref{eq:osc:cancellation:1} and \eqref{eq:osc:cancellation:2}, we arrive at 
\begin{align}
&w_{q+1}^p \otimes w_{q+ 1}^p - d_{q+1}^p \otimes d_{q+1}^p + \mathring{R}_{\ell}^u \notag \\
&= \rho_u \Id +  \sum_{ k \in \Lambda_u} a_{(k)}^2 \mathbb{P}_{\neq 0}(\phi_{(k)}^2)k_1 \otimes k_1  + \sum_{k \in \Lambda_B}a_{(k)}^2 \mathbb{P}_{\neq 0}(\phi_{(k)}^2)(k_1 \otimes k_1 - k_2 \otimes k_2) \notag\\
 &\quad + \sum_{ k \neq k' \in \Lambda_u} a_{(k)}a_{(k')}\phi_{(k)}\phi_{(k')}k_1\otimes k_1' 
 + \sum_{ k \neq k' \in \Lambda_B} a_{(k)}a_{(k')}\phi_{(k)}\phi_{(k')} (k_1\otimes k_1' - k_2\otimes k_2') \notag\\
 &\quad + \sum_{k \in \Lambda_u, k' \in \Lambda_B}  a_{(k)}a_{(k')}\phi_{(k)}\phi_{(k')}(k_1\otimes k_1'  + k_1'\otimes k_1)    \,.
 \label{eq:osc:cancellation}
\end{align}
The calculation in \eqref{eq:osc:cancellation} motivates the definition of $\mathring{G}^B$: due to the fact that $w_{q+1}^p$ needs more wavevectors than $d_{q+1}^p$, we get an extra self-interaction term in the expansion of $w_{q+1}^p\otimes w_{q+1}^p$ that is too large to go into the next Reynolds stress so must be cancelled completely.  

Note that as a consequence of the definitions~\eqref{principal ptb}, the estimates \eqref{linear intermittent  estimates},  \eqref{mag amp estimates}, \eqref{vel amp estimates}, and the parameter inequality $\ell^{-10} \ll \lambda_{q+1}$ we have
\begin{align}
\label{principal c1 est}
\norm{ w_{q+1}^p }_{C_{x,t}^1 }  +   \norm{ d_{q+1}^p }_{C_{x,t}^{1}} \lesssim \ell^{-2} \lambda_{q+1} r^{-\frac{1}{2}}.
\end{align}

\subsection{Incompressibility Correctors}
Due to the spatial dependence of the amplitudes $a_{(k)}$,  the principal parts of the perturbation, $w_{q+1}^p$ and $d_{q+1}^p$, are no longer divergence free. To fix this, we define \textit{incompressibility correctors} analogously to \cite{BCV18}. First define 
\begin{equation}
\label{corrector vector}
 \quad W_k^c := \frac{1}{N_{\Lambda}^2 \lambda_{q+1}^2} \Phi_{(k)} k_1 \, ,  \qquad D_k^c:= \frac{1}{N_{\Lambda}^2 \lambda_{q+1}^2} \Phi_{(k)} k_2 \, .
\end{equation}
Then we define the incompressibility correctors  
\begin{align*}
w_{q+1}^c &:=  \sum_{k\in \Lambda_u  }\curl (\nabla a_{(k)} \times W_k^c) + \nabla a_{(k)} \times \curl W_k^c +  \sum_{k\in \Lambda_B  }\curl (\nabla a_{(k)} \times W_k^c) + \nabla a_{(k)} \times \curl W_k^c\\
d_{q+1}^c &:=  \sum_{k\in \Lambda_B }\curl (\nabla a_{(k)} \times D_k^c) + \nabla a_{(k)} \times \curl D_k^c \,. \end{align*}

With this definition we see that
\begin{align}
\label{div free velocity}
\curl  \curl \left( \sum_{k \in \Lambda_u } a_{(k)} W_k^c  + \sum_{k \in \Lambda_B } a_{(k)} W_k^c   \right)  
&= \sum_{k \in \Lambda_u } a_{(k)}W_k  + \curl (\nabla a_{(k)} \times W_k^c) + \nabla a_{(k)} \times\curl W_k^c \notag\\
&+ \sum_{k \in \Lambda_B } a_{(k)}W_k  + \curl (\nabla a_{(k)} \times W_k^c) + \nabla a_{(k)} \times\curl W_k^c \notag \\
&= w_{q+1}^p + w_{q+1}^c 
\end{align}
and
\begin{align}
\label{div free magnetic}
\curl  \curl \left( \sum_{k \in \Lambda_B} a_{(k)} D_k^c \right)  &= \sum_{k \in \Lambda_B } a_{(k)}D_k  + \curl (\nabla a_{(k)} \times D_k^c) + \nabla a_{(k)} \times \curl D_k^c = d_{q+1}^p + d_{q+1}^c \,.
\end{align}
From \eqref{div free velocity} and \eqref{div free magnetic} we deduce that $\div (w_{q+1}^p + w_{q +1}^c) = \div (d_{q+1}^p + d_{q +1}^c) =  0$, which justifies the definitions of the incompressibility correctors.

Using  \eqref{corrector vector}, \eqref{mag amp estimates}, and \eqref{linear intermittent  estimates}, and the fact that $\ell^{-5} \ll \lambda_{q+1} $ we obtain for any $p\in [1, \infty]$ 
\begin{align}
\label{lp mag corrector estimate}
\norm{ d_{q+1}^c }_{L^p} &\leq  \sum_{ k \in \Lambda_B} \norm{ \curl (\nabla a_{(k)} \times D_k^c) + \nabla a_{(k)} \times \curl   D_k^c  }_{L^p} \notag \\
 &\leq \sum_{k \in \Lambda_B }  \norm{D_k^c \, \nabla^2 a_{(k)} }_{L^p} + \norm{ \nabla a_{(k)} \cdot \nabla D_k^c }_{L^p} + \norm{ \nabla a_{(k)} \times \curl  D_k^c }_{L^p} \notag \\
 &\lesssim \norm{ a_{(k)}}_{C_{x,t}^1} \norm{ D_k^c}_{W^{1,p}} + \norm{ a_{(k)} }_{C_{x,t}^2} \norm{ D_k^c }_{L^p} \lesssim \ell^{-7}r^{\frac{1}{p} - \frac{1}{2}} \lambda_{q+1}^{-1} + \lambda_{q+1}^{-2}r^{\frac{1}{p} -\frac{1}{2}}
  \ell^{-12} \notag \\
 &\lesssim  \ell^{-7}r^{\frac{1}{p} - \frac{1}{2}} \lambda_{q+1}^{-1} \,.
\end{align}
Using \eqref{vel amp estimates}  for $k \in \Lambda_u $ we also have that 
\begin{equation}
\label{lp vel corrector estimate}
\norm{ w_{q+1}^c }_{L^p} \lesssim \ell^{-12}r^{\frac{1}{p} - \frac{1}{2}}\lambda_{q+1}^{-1}\,.
\end{equation}
Thus, by \eqref{vel amp estimates}, \eqref{mag amp estimates},  \eqref{linear intermittent  estimates}, and $\ell^{-10} \ll \lambda_{q+1} $ we obtain 
\begin{equation}
\label{corrector c1 est}
\norm{w_{q+1}^c}_{C_{x,t}^1} \lesssim \ell^{-12}r^{-\frac{1}{2}}\,  \qquad \mbox{and} \qquad \norm{d_{q+1}^c}_{C_{x,t}^1} \lesssim \ell^{-7}r^{-\frac{1}{2}} \,.
\end{equation}

Lastly, we define the  velocity and magnetic perturbations:
\begin{subequations}
\label{perturbation}
\begin{align}
 w_{q+1} &:= w_{q+1}^p + w_{q+1}^c
\label{velocity perturbation}\\
 d_{q+1} &:= d_{q+1}^p + d_{q+1}^c
\label{magnetic perturbation}
\end{align}
\end{subequations}
and the next iterate:
\begin{subequations}
\label{q+1 iterate}
\begin{align}
 v_{q+1} &:= v_{\ell} + w_{q+1}
\label{q+1 velocity}\\
 B_{q+1} &:= B_{\ell}+ d_{q+1} \,.
\label{q+1 magnetic}
\end{align}
\end{subequations}

\subsection{$L^p$ Decorrelation}
In order to verify the inductive estimates on the perturbations $w_{q+1}$ and $d_{q+1}$ we will need the  $L^p$ Decorrelation Lemma from~\cite{BV}, which we record here for convenience. 
\begin{lemma}[$L^p$ Decorrelation]
\label{Decorrelation}
Fix integers N, $\kappa \geq 1$ and let $\zeta > 1$ be such that 
\begin{equation*}
\frac{2 \pi \sqrt{3} \zeta}{\kappa} \leq \frac{1}{3}  \text{ and } \zeta^4 \frac{(2\pi \sqrt{3}\zeta)^N }{\kappa^N} \leq 1
\end{equation*}
Let $p \in \{ 1,2\} $, and let $f$ be a $\mathbb{T}^3$-periodic function such that there exists a constant $C_f > 0$ such that 
$$
\norm{D^j f }_{L^p} \leq C_f \zeta^j 
$$
holds for all $0 \leq j \leq N + 4$.  In addition, let $g$ be a $(\mathbb{T}/\kappa)^3-periodic$ function. Then we have that 
$$
\norm{fg }_{L^p} \leq C_f C_* \norm{ g}_{L^p},
$$
where $C_*$ is a universal constant. 
\end{lemma}
We will apply this lemma with $f= a_{(k)}$, $g = \phi_{(k)}$, $\kappa = r\lambda_{q+1}$, $N = 1$ and $p = 2$. The choice of $C_f$ and $\zeta$  depends on the wavevector set.
For $k \in \Lambda_B$, using \eqref{mag amp L2}, \eqref{mag amp estimates}, and that $\ell \leq \delta_{q+1}$ we have for $j \geq 0$
\begin{equation*}
\norm{D^ja_{(k)} }_{L^2} \leq \frac{\delta_{q+1}^{\frac{1}{2}}}{3 C_* (8 \pi^3)^{\frac{1}{2}} |\Lambda_B|} \ell^{-8j}\,, \qquad k \in \Lambda_B \,.
\end{equation*} 
For $k \in \Lambda_u$, using \eqref{vel amp L2 estimate} and \eqref{vel amp estimates} 
gives 
\begin{equation*}
\norm{D^ja_{(k)} }_{L^2} \leq \frac{\delta_{q+1}^{\frac{1}{2}}}{3 C_* (8 \pi^3)^{\frac{1}{2}}|\Lambda_u|} \ell^{-13j}\, ,\qquad k \in \Lambda_u \,.
\end{equation*}
Thus we can take $C_f = \frac{\delta_{q+1}^{\frac{1}{2}}}{3 C_* (8 \pi^3)^{\frac{1}{2}}|\Lambda_B|}$ and $\zeta = \ell^{-8}$ for $k \in \Lambda_B$ and  $C_f = \frac{\delta_{q+1}^{\frac{1}{2}}}{3 C_* (8 \pi^3)^{\frac{1}{2}} |\Lambda_u|}$  with $\zeta = \ell^{-13} $ for $k \in \Lambda_u$. We are justified in applying the decorrelation lemma with the above chosen parameters because $\ell^{-65} \ll r \lambda_{q+1} = \lambda_{q+1}^{\frac{1}{4}}$ which is the most restrictive condition coming from our choice of parameters.

Applying Lemma~\ref{Decorrelation} gives 
\begin{align}
\label{Lp decorr mag}
\norm{ a_{k, B} \phi_{(k)}}_{L^2} &\leq \frac{\delta_{q+1}^{\frac{1}{2}}}{3(8 \pi^3)^{\frac{1}{2}} |\Lambda_B|} \norm{ \phi_{(k)} }_{L^2} = \frac{\delta_{q+1}^{\frac{1}{2}}}{3 |\Lambda_B|}\\
\label{Lp decorr vel}
\norm{ a_{k, u} \phi_{(k)}}_{L^2} &\leq \frac{\delta_{q+1}^{\frac{1}{2}}}{3(8 \pi^3)^{\frac{1}{2}} |\Lambda_u|} \norm{ \phi_{(k)} }_{L^2} = \frac{\delta_{q+1}^{\frac{1}{2}}}{3 |\Lambda_u|}
\end{align}
since $\phi_{(k)}^2$  was normalized to have unit average over $\T^3$. 

\subsection{Verification of inductive estimates}

Using \eqref{lp mag corrector estimate} and \eqref{Lp decorr mag}  we can verify inductive estimates \eqref{mag inductive assumptions} and \eqref{vel inductive assumptions}.  For the magnetic increment we have the bound
\begin{equation}
\label{mag perturbation l2 }
\norm{d_{q+1} }_{L^2} \leq  \norm{ d_{q+1}^p}_{L^2} +\norm{ d_{q+1}^c }_{L^2} \leq \norm{ \sum_{ k \in \Lambda_B} a_{(k)} D_{(k)} }_{L^2} + \norm{ d_p^c }_{L^2} \leq \frac{1}{3}\delta_{q+1}^{\frac{1}{2}} + \ell^{-8} \lambda_{q+1}^{-1} \leq \frac{1}{2}\delta_{q+1}^{\frac{1}{2}}
\end{equation}
where we used an extra power of $\ell$ to absorb any implicit constants coming from \eqref{lp mag corrector estimate} and  that $\lambda_{q+1}^{-1} \ll \ell^{8}\delta_{q+1} $ in the last inequality.
Similarly, for the velocity we have
\begin{align}
\label{vel perturbation l2}
\norm{w_{q+1}}_{L^2} &\leq \norm{w_{q+1}^p }_{L^2} + \norm{ w_{q+1}^c }_{L^2} \notag \\
& \leq \norm{ \sum_{k \in \Lambda_u }a_{(k)} W_{(k)} + \sum_{k \in \Lambda_B }a_{(k)} W_{(k)} }_{L^2}  + \norm{ w_{q+1}^c }_{L^2} \notag \\  
&\leq \sum_{k \in \Lambda_u} \norm{a_{(k)} W_{(k)} }_{L^2} + \sum_{k \in \Lambda_B} \norm{a_{(k)} W_{(k)} }_{L^2} + \norm{ w_{q+1}^c }_{L^2} \notag \\
&\leq \frac{\delta_{q+1}^{\frac{1}{2}}  }{3} + \frac{\delta_{q+1}^{\frac{1}{2}}  }{3} + \ell^{-13}\lambda_{q+1}^{-1} \leq \frac{3}{4} \delta_{q+1}^{\frac{1}{2}} \,.
\end{align}

Applying standard mollification estimates, using \eqref{mag inductive assumptions}, \eqref{vel inductive assumptions}, and $\eta b - 2 \gg b\beta $
\begin{equation}
\label{mag mollification error}
\norm{B_q - B_{\ell} }_{L^2} \lesssim \norm{ B_q - B_{\ell} }_{C^0} \lesssim \ell \norm{B_q }_{C_{x,t}^1} \lesssim \ell \lambda_{q}^2 \ll \delta_{q+1}^{\frac{1}{2}}
\end{equation}
and 
\begin{equation}
\label{vel mollification error}
\norm{u_q - u_{\ell} }_{L^2} \lesssim \norm{ u_q - u_{\ell} }_{C^0} \lesssim \ell \norm{u_q }_{C_{x,t}^1} \lesssim \ell \lambda_{q}^2 \ll \delta_{q+1}^{\frac{1}{2}} \,.
\end{equation}

Combining \eqref{mag mollification error}, \eqref{mag perturbation l2 },  \eqref{vel mollification error}, and \eqref{vel perturbation l2} for the magnetic field and velocity respectively we obtain
\begin{align*}
&\norm{B_{q} - B_{q+1} }_{L^2} \leq \norm{ B_{q} - B_{\ell}  }_{L^2} + \norm{B_{\ell}  - B_{q+1}}_{L^2} \leq \frac{1}{2}\delta_{q+1}^{\frac{1}{2}} + \norm{d_{q+1}}_{L^2} \leq \delta_{q+1}^{\frac{1}{2}}\\
&\norm{u_{q} - u_{q+1} }_{L^2} \leq \norm{ u_{q} - u_{\ell}  }_{L^2} + \norm{u_{\ell}  - u_{q+1}}_{L^2} \leq \frac{1}{4}\delta_{q+1}^{\frac{1}{2}} + \norm{w_{q+1}}_{L^2} \leq \delta_{q+1}^{\frac{1}{2}}
 \end{align*}
as desired.

Now we check the $L^2$ norm: 
\begin{equation*}
\norm{ B_{q+1} }_{L^2} = \norm{B_{\ell} + d_{q+1} }_{L^2} \leq \norm{B_{\ell}}_{L^2} + \norm{d_{q+1} }_{L^2} \leq 1 - \delta_{q}^{\frac{1}{2}} + \delta_{q+1}^{\frac{1}{2}} \leq 1 - \delta_{q+1}^{\frac{1}{2}}
\end{equation*}
where we used that $2\delta_{q+1}^{\frac{1}{2}} \leq \delta_{q}^{\frac{1}{2}}$. 
The same reasoning shows that $\norm{ u_{q+1} }_{L^2} \leq 1 - \delta_{q+1}^{\frac{1}{2}}$  as well.

We finish by checking the $C_{x,t}^1$ estimate 
 for the velocity and magnetic field at level $q+1$:  using the parameter inequality $\ell^{-1} \ll r^{-1} \ll \lambda_{q+1}$, and the bounds \eqref{principal c1 est}, \eqref{corrector c1 est},  we have 
\begin{equation*}
\| d_{q+1} \|_{C_{x,t}^1} \leq \|d_{q+1}^p \|_{C_{x,t}^1} + \| d_{q+1}^c \|_{C_{x,t}^1} \lesssim \ell^{-2}\lambda_{q+1} r^{-\frac{1}{2}} + \ell^{-7}r^{-\frac{1}{2}} \lesssim \ell^{-2}\lambda_{q+1} r^{-\frac{1}{2}} \leq \lambda_{q+1}^2
\end{equation*}
and
\begin{equation*}
\| w_{q+1} \|_{C_{x,t}^1} \leq \|w_{q+1}^p \|_{C_{x,t}^1} + \| w_{q+1}^c \|_{C_{x,t}^1} \lesssim \ell^{-2}\lambda_{q+1} r^{-\frac{1}{2}} + \ell^{-12}r^{-\frac{1}{2}} \lesssim \ell^{-2}\lambda_{q+1} r^{-\frac{1}{2}} \leq \lambda_{q+1}^2 \, .
\end{equation*}

\section{Reynolds and Magnetic Stress}
\label{stress}
\subsection{Symmetric Inverse divergence}
 In order to define the Reynolds and magnetic stress we need an inverse divergence operator that acts on mean-free vector fields. For the Reynolds stress it suffices to use the inverse-divergence operator from~\cite{DLSZ13}:
\begin{equation*}
 (\mathcal{R}v)^{kl} = \partial_k \Delta^{-1} v^l + \partial_l \Delta^{-1} v^k - \frac{1}{2}(\delta_{kl} + \partial_k \partial_l \Delta^{-1})\div \Delta^{-1} v 
\end{equation*}
where $k, \ell \in \{1,2,3 \}$. The operator $\mathcal{R}$ returns a symmetric, trace-free matrix and satisfies the following key identity for mean-free vector fields: $\div \mathcal{R}(v) = v$. Note that $|\nabla| \mathcal{R}$ is a Calderon-Zygmund operator.

\subsection{Skew-Symmetric Inverse divergence} 
Unlike in previous convex integration schemes, we will also need an inverse divergence that returns skew-symmetric matrices as opposed to symmetric trace-free ones. We will denote this operator as $\mathcal{R}^B$. We want 
$\div  \mathcal{R}^B (f) = f$ where $f: \mathbb{R}^3 \to \mathbb{R}^3$ and $\mathcal{R}^B(f) = - (\mathcal{R}^B(f))^{\top}$. If we define
\begin{equation*}
(\mathcal{R}^Bf)_{ij} :=  \varepsilon_{ijk} (-\Delta)^{-1}(\curl f)_k
\end{equation*}
where $\varepsilon_{ijk}$ is the Levi-Civita tensor and $\div f = 0$, then a direct calculation of the  divergence (contracting along the second index) shows that $\div  \mathcal{R}^B (f) = f$. Again, $|\nabla| \mathcal{R}^B $ is a Calderon-Zygmund operator.

\subsection{Decomposition of the stresses} 
Our goal is now to show that  the stresses $\mathring{R}_{q+1}^u$ and $ \mathring{R}_{q+ 1}^{B}$  satisfy \eqref{vel inductive assumptions} and \eqref{mag inductive assumptions}. However, we must first determine  $\mathring{R}_{q+1}^u$ and $\mathring{R}_{q+1}^B$. To do this, consider the equation satisfied by $(u_{q+1}, B_{q+1})$: 
\begin{align}
\div\mathring{R}_{q+1}^u - \nabla p_{q+1} &= \underbrace{\partial_t w_{q+1} + \div (v_{\ell} \otimes w_{q+1} + w_{q+ 1} \otimes v_{\ell} - B_{\ell} \otimes d_{q+1} - d_{q+1} \otimes B_{\ell}) }_{\div \mathring{R}_{lin}^u + \nabla p_{lin} } \notag\\ 
                                                &+ \underbrace{\div (w_{q+1}^p \otimes w_{q+1}^p - d_{q+1}^p \otimes d_{q+1}^p +  \mathring{R}_{\ell}^u)}_{\div\mathring{R}_{osc}^u + \nabla p_{osc}} \notag \\ 
                                                &+ \underbrace{\div(w_{q+1} \otimes w_{q+1}^c + w_{q+1}^c \otimes w_{q+1}^p - d_{q+1} \otimes d_{q+1}^c - d_{q+1}^c \otimes d_{q+1}^p )}_{ \div \mathring{R}_{corr}^u + \nabla p_{corr}} \notag \\ 
                                                 &+ \div \mathring{R}_{comm}^u - \nabla p_{\ell} 
                                                 \label{velocity reynolds decomp}  
\end{align}
and
\begin{align}
\div\mathring{R}_{q+1}^B &= \underbrace{\partial_t d_{q+ 1}  + \div  (u_{\ell} \otimes d_{q+1} + w_{q +1} \otimes B_{\ell} - B_{\ell} \otimes w_{q+1} - d_{q + 1} \otimes u_{\ell})}_{\div \mathring{R}_{lin}^B } \notag \\
                    &+ \underbrace{\div (w_{q+ 1}^p \otimes  d_{q+ 1}^p - d_{q+ 1}^p \otimes w_{q+1}^p + \mathring{R}_{\ell}^B)}_{ \div \mathring{R}_{osc}^B } \notag \\
                    &+ \underbrace{\div(w_{q+1}^c \otimes d_{q+1} - d_{q+1} \otimes w_{q+1}^c  + w_{q+1}^p \otimes d_{q+1}^c - d_{q+1}^c\otimes w_{q+1}^p   )}_{\div \mathring{R}_{corr}^B  } \notag  \\
                    &+ \div \mathring{R}_{comm}^B \, .
                     \label{magnetic reynolds decomp} 
\end{align}
Applying the symmetric and skew-symmetric inverse divergence operators allows us to define the different parts of the Reynolds and magnetic stresses as follows:
\begin{align}
    \mathring{R}_{lin}^B &= \mathcal{R}^B(\partial_t d_{q+1}) +  u_{\ell} \otimes d_{q+1} - d_{q + 1} \otimes u_{\ell} + w_{q +1} \otimes B_{\ell} - B_{\ell} \otimes w_{q+1}
    \label{lin mag stress}\\
    \mathring{R}_{corr}^B &= w_{q+1}^c \otimes d_{q+1} - d_{q+1} \otimes w_{q+1}^c  + w_{q+1}^p \otimes d_{q+1}^c - d_{q+1}^c\otimes w_{q+1}^p
    \label{corrector mag stress}
\end{align}
and
\begin{align}
    \mathring{R}_{lin}^u &= \mathcal{R}(\partial_t w_{q+1}) + v_{\ell} \mathring{\otimes} w_{q+1} + w_{q+ 1} \mathring{\otimes} v_{\ell} - B_{\ell} \mathring{\otimes} d_{q+1} - d_{q+1} \mathring{\otimes} B_{\ell}
     \label{lin vel Reynolds}\\
    \mathring{R}_{corr}^u &= w_{q+1} \mathring{\otimes} w_{q+1}^c + w_{q+1}^c \mathring{\otimes} w_{q+1}^p - d_{q+1} \mathring{\otimes} d_{q+1}^c - d_{q+1}^c \mathring{\otimes} d_{q+1}^p  
    \label{corrector vel stress} \, .
\end{align}
The associated pressure terms are defined as 
$ p_{lin}  =  2v_{\ell} \cdot w_{q+1}  - 2B_{\ell} \cdot d_{q+1}$ and 
$p_{corr}= w_{q+1} \cdot w_{q+1}^c + w_{q+1}^c \cdot w_{q+1}^p - d_{q+1} \cdot d_{q+1}^c - d_{q+1}^c \cdot d_{q+1}^p$.
In order to determine the equation for $\mathring{R}_{osc}^B$, we use \eqref{mag oscillation cancellation calculation} and the fact that $k_1 \cdot \nabla \phi_{(k)} = k_2\cdot \nabla \phi_{(k)}=0$, and obtain
\begin{align}
&\div(w_{q+ 1}^p \otimes  d_{q+ 1}^p - d_{q+ 1}^p \otimes w_{q+1}^p + \mathring{R}_{\ell}^B) \notag\\
&= \sum_{k \in \Lambda_B} \nabla (a_{(k)}^2) \mathbb{P}_{\neq 0}(\phi_{(k)}^2)(k_1 \otimes k_2 - k_2 \otimes k_1) 
+ \div \Biggl( \sum_{ k \neq k' \in \Lambda_B} a_{(k)}a_{(k')}\phi_{(k)}\phi_{(k')}(k_1\otimes k_2' - k_2'\otimes k_1) \Biggr)
\notag\\ 
&\ +  \div \Biggl( \sum_{ k \in \Lambda_u, k' \in \Lambda_B} a_{(k)}a_{(k')}\phi_{(k)}\phi_{(k')}(k_1\otimes k_2' - k_2'\otimes k_1) \Biggr) \, .
\label{eq:RB:osc}
\end{align}
Here and throughout the paper we use the notation $\nabla f \, (\ell \otimes \ell')$ to denote the contraction on the second component of the tensor, namely $\ell (\ell'\cdot\nabla)f$.
Similarly, to find $\mathring{R}_{osc}^u$ and $p_{osc}$ we appeal to \eqref{eq:osc:cancellation} and apply the divergence operator, to arrive at
\begin{align}
 &\div \left( w_{q+1}^p \otimes w_{q+1}^p - d_{q+1}^p \otimes d_{q+1}^p +  \mathring{R}_{\ell}^u \right) \notag\\
 &= \nabla p_{osc} + \sum_{k \in \Lambda_u}  \nabla (a_{(k)}^2) \mathbb{P}_{\neq 0}(\phi_{(k)}^2)k_1 \otimes k_1   + \sum_{k \in \Lambda_B}  \nabla(a_{(k)}^2) \mathbb{P}_{\neq 0}(\phi_{(k)}^2)(k_1 \otimes k_1 - k_2 \otimes k_2)   \notag\\
 &\ + \div \Biggl( \sum_{ k \neq k' \in \Lambda_u} a_{(k)}a_{(k')}\phi_{(k)}\phi_{(k')} k_1\mathring{\otimes} k_1' + \sum_{ k \neq k' \in \Lambda_B} a_{(k)}a_{(k')}\phi_{(k)}\phi_{(k')} (k_1\mathring{\otimes} k_1' - k_2 \mathring{\otimes} k_2') \Biggr)\notag\\
  &\ + \div \Biggl(    \sum_{k \in \Lambda_u, k' \in \Lambda_B}  a_{(k)}a_{(k')}\phi_{(k)}\phi_{(k')}(k_1\mathring{\otimes} k_1'  + k_1' \mathring{\otimes} k_1)  \Biggr)
\, ,
\label{eq:Ru:osc}
\end{align}
where
\begin{align*}
p_{osc} &= \rho_u  +  \sum_{ k \neq k' \in \Lambda_u} a_{(k)}a_{(k')}\phi_{(k)}\phi_{(k')} k_1 \cdot k_1' + \sum_{ k \neq k' \in \Lambda_B} a_{(k)}a_{(k')}\phi_{(k)}\phi_{(k')} (k_1 \cdot k_1' - k_2 \cdot k_2') \notag\\
&\qquad +  2 \sum_{k \in \Lambda_u, k' \in \Lambda_B}  a_{(k)}a_{(k')}\phi_{(k)}\phi_{(k')} k_1 \cdot k_1' 
\end{align*}
Therefore, from \eqref{eq:RB:osc} we have that the magnetic oscillation stress is given by
\begin{align}
\mathring{R}_{osc}^B &=  
\sum_{k \in \Lambda_B}\mathcal{R}^B \left( \nabla (a_{(k)}^2) \mathbb{P}_{\neq 0}(\phi_{(k)}^2)  (k_1 \otimes k_2 - k_2 \otimes k_1)\right)
\notag\\
&\qquad + \sum_{ k \neq k' \in \Lambda_B} a_{(k)}a_{k', B}\phi_{(k)}\phi_{(k')}(k_1\otimes k_2' - k_2'\otimes k_1) \notag\\
&\qquad  + \sum_{ k \in \Lambda_u, k' \in \Lambda_B } a_{(k)}a_{(k')}\phi_{(k)}\phi_{(k')}(k_1\otimes k_2' - k_2'\otimes k_1) 
    \label{mag osc stress} \, ,
\end{align}
while from \eqref{eq:Ru:osc} we deduce that the Reynolds oscillation stress is defined as
\begin{align}
\mathring{R}_{osc}^u &= \sum_{k \in \Lambda_u} \mathcal{R}\left(  \nabla (a_{(k)}^2) \mathbb{P}_{\neq 0}(\phi_{(k)}^2)k_1 \otimes k_1 \right) + \sum_{k \in \Lambda_B}\mathcal{R} \left( \nabla(a_{(k)}^2) \mathbb{P}_{\neq 0}(\phi_{(k)}^2)(k_1 \otimes k_1 - k_2 \otimes k_2) \right) \notag\\
&\qquad+  \sum_{ k \neq k' \in \Lambda_u} a_{(k)}a_{(k')}\phi_{(k)}\phi_{(k')} k_1\mathring{\otimes} k_1' + \sum_{ k \neq k' \in \Lambda_B} a_{(k)}a_{(k')}\phi_{(k)}\phi_{(k')} (k_1\mathring{\otimes} k_1' - k_2 \mathring{\otimes} k_2')  \notag\\
  &\qquad +   \sum_{k \in \Lambda_u, k' \in \Lambda_B}  a_{(k)}a_{(k')}\phi_{(k)}\phi_{(k')}(k_1\mathring{\otimes} k_1'  + k_1' \mathring{\otimes} k_1)   
 \label{vel osc stress}\, .
\end{align}

In conclusion, we note that the  pressure at level $q+1$ is given by
$p_{q+1} := p_{\ell} - p_{lin} - p_{osc} - p_{corr}$, 
while the  magnetic and Reynolds stresses are given respectively by
\begin{subequations}
\label{stress q+1}
\begin{align}
    \mathring{R}_{q+1}^B &= \mathring{R}_{lin}^B + \mathring{R}_{osc}^B + \mathring{R}_{corr}^B + \mathring{R}_{comm}^B
\label{magnetic stress q+1}\\
\mathring{R}_{q+1}^u &= \mathring{R}_{lin}^u + \mathring{R}_{osc}^u + \mathring{R}_{corr}^u + \mathring{R}_{comm}^u.
\label{velocity stress q+1}    
\end{align}
\end{subequations}

\subsection{Estimates for the magnetic  stress}
In order to estimate the stresses in $L^1$, since Calderon-Zygmund operators are not bounded on $L^1$, we fix an integrability parameter $p$ sufficiently close to $1$, which we will use whenever we have a stress term that involves a Calderon-Zygmund operator.

\subsubsection{Linear Error} 
We first estimate the time derivative term in \eqref{lin mag stress}.
By \eqref{div free magnetic} we have $\partial_t d_{q +1 } = \curl  \curl ( \sum_{ \Lambda_B} \partial_t( a_{(k)} D_k^c)) = \curl  \curl ( \sum_{ \Lambda_B} \partial_t a_{(k)} D_k^c)$. Therefore using \eqref{mag amp estimates}, the definition of $D_k^c$   in \eqref{corrector vector}, and \eqref{linear estimate lemma} we have
\begin{align}
\label{mag time derivative}
\| \mathcal{R}^B(\partial_t d_{q+1}) \|_{L^1} \les \| \mathcal{R}^B(\partial_t d_{q+1}) \|_{L^p} &\lesssim  \sum_{k \in \Lambda_B}\| \mathcal{R}^B \curl  \curl  (\partial_t a_{(k)} D_k^c) \|_{L^p} \notag\\ 
& \lesssim \sum_{k \in \Lambda_B}\| \curl (\partial_t a_{(k)} D_k^c)  \|_{L^p} \notag\\
&\lesssim \sum_{k \in \Lambda_B}\|  a_{(k)} \|_{C_{x,t}^2}\| D_k^c \|_{W^{1,p}} \notag \\
&\lesssim \ell^{-12} \lambda_{q+1}^{-1}  
\end{align}
where we used the fact that $1\leq p \leq 2$ to remove the (good) $r$ factor from the $\|\nabla D_k^c \|_{L^p}$ estimate. 

Next we estimate the high-low interaction terms present in \eqref{magnetic reynolds decomp}.
First we write $d_{q+1} = d_{q+1}^p + d_{q+1}^c $ so we have $ u_{\ell} \otimes d_{q+1} = u_{\ell} \otimes d_{q+1}^p + u_{\ell} \otimes d_{q+1}^c $. We will only show how to estimate one term since the other terms can be handled similarly. By \eqref{vel inductive assumptions}, regularizing properties of mollification, \eqref{lp mag corrector estimate},  and \eqref{linear intermittent  estimates} we have
\begin{align}
\label{mag cross terms}
\norm{ u_{\ell} \otimes d_{q+1} }_{L^1} &\leq \norm{ u_{\ell} \otimes d_{q+1}^p }_{L^1} + \norm{ u_{\ell} \otimes d_{q+1}^c }_{L^1} \notag\\
&\leq \norm{u_{\ell}}_{C^{0}} \norm{d_{q+1}^p }_{L^1} + \norm{ u_{\ell} }_{L^2}\norm{ d_{q+1}^c }_{L^2} \notag \\
&\lesssim \ell^{-\frac{3}{2}}  \norm{d_{q+1}^p }_{L^1} + \norm{ d_{q+1}^c }_{L^2} \notag \\
&\lesssim \ell^{-\frac{3}{2}} \ell^{-2} r^{\frac{1}{2}} + \ell^{-7} \lambda_{q+1}^{-1} \notag \\
&\lesssim \ell^{-4}r^{\frac{1}{2}}
\end{align}
where we used that $\lambda_{q+1}^{-1} \ll r \ll \ell $. The same estimate also holds for the term $w_{q +1} \otimes B_{\ell}$. Therefore,  
\begin{align}
\label{magnetic linear estimate}
\norm{\mathring{R}_{lin}^B }_{L^1} &\lesssim \norm{ \mathcal{R}^B(\partial_t d_{q+1}) }_{L^p} + \norm{u_{\ell} \otimes d_{q+1} - d_{q + 1} \otimes u_{\ell} + w_{q +1} \otimes B_{\ell} - B_{\ell} \otimes w_{q+1} }_{L^1} \notag \\
&\lesssim \ell^{-12}\lambda_{q+1}^{-1} + \ell^{-4}r^{\frac{1}{2}} \notag \\
&\lesssim \ell^{-4} r^{\frac{1}{2}}.
\end{align}

\subsubsection{Oscillation Error}

In order to estimate the magnetic oscillation stress  we use \eqref{mag osc stress} to decompose it into two parts:
\begin{align*}
\mathring{R}_{osc}^B &=  E_1^B + E_2^B
\end{align*}
where 
\begin{align*}
    E_1^B &:= \sum_{k \in \Lambda_B}\mathcal{R}^B \left( \nabla (a_{(k)}^2) \mathbb{P}_{\neq 0}(\phi_{(k)}^2)  (k_1 \otimes k_2 - k_2 \otimes k_1)\right)  \\
    E_2^B &:= \!\!\! \sum_{ k \neq k' \in \Lambda_B} a_{(k)}a_{k', B}\phi_{(k)}\phi_{(k')}(k_1\otimes k_2' - k_2'\otimes k_1) \notag   + \!\!\!  \sum_{ k \in \Lambda_u,  k' \in \Lambda_B} a_{(k)}a_{(k')}\phi_{(k)}\phi_{(k')}(k_1\otimes k_2' - k_2'\otimes k_1) \, .
\end{align*}
First note that since $\div  \left(a_{(k)}^2 \mathbb{P}_{\neq 0}(\phi_{(k)}^2)(k_1 \otimes k_2 -k_2 \otimes k_1) \right) = \nabla (a_{(k)}^2) \mathbb{P}_{\neq 0}(\phi_{(k)}^2) (k_1 \otimes k_2 - k_2 \otimes k_1)$, we can conclude that $\nabla (a_{(k)}^2) \mathbb{P}_{\neq 0}(\phi_{(k)}^2) (k_1 \otimes k_2 - k_2 \otimes k_1)$ is mean free. A calculation also shows that
 $\div \div  (a_{(k)}^2 \mathbb{P}_{\neq 0}(\phi_{(k)}^2)(k_1 \otimes k_2 - k_2 \otimes k_1))  = 0$
so $E_1^B$ is well-defined.
Therefore it suffices to estimate  $E_1^B$ and $E_2^B$ individually. 
For $E_1^B$, we note that since $\phi_{(k)}$ is $\lambda_{q+1}r$ periodic, so is $\phi_{(k)}^2$. Therefore the minimal active frequency in $\mathbb{P}_{\neq 0} \phi_{(k)}^2$ is $\lambda_{q+1} r$; we have that $\mathbb{P}_{\neq 0}(\phi_{(k)}^2) = \mathbb{P}_{\geq (\lambda_{q+1} r/2)}(\phi_{(k)}^2)$. This allows us to exploit the frequency separation between $\nabla (a_{(k)}^2)$ and $\phi_{(k)}^2$ and gain a factor of $\lambda_{q+1} r$ from the application of  $\mathcal{R}^{B}$.
To be precise, we recall Lemma B.1 from \cite{BV}: 

\begin{lemma}
\label{commutator estimate}

Fix parameters $1 \leq \zeta  < \kappa, p \in (1,2]$, and assume there exists an $L \in \mathbb{N}$ such that 
\begin{equation*}
    \zeta^L \leq \kappa^{L-2}
\end{equation*}
Let $ a \in C^L(\mathbb{T}^3)$ be such that there exists $C_a > 0$ with
\begin{equation*}
    \norm{ D^j a}_{C^0} \leq C_a \zeta^j
\end{equation*}
for all $0 \leq j \leq L$. Assume also that $f \in L^{p}(\mathbb{T}^3)$ is such that $\int_{\mathbb{T}^3}a(x) \mathbb{P}_{\geq \kappa} f(x) dx = 0 $. Then we have 
\begin{equation*}
    \norm{ |\nabla|^{-1} (a \mathbb{P}_{\geq \kappa}f )}_{L^P} \lesssim C_a \frac{\norm{f }_{L^p}}{\kappa}
\end{equation*}
where the implicit constant depends only on $p$ and $L$.
\end{lemma}

Using \eqref{G estimates} we see that we can apply Lemma~\ref{commutator estimate} with $a = \nabla (a_{(k)}^2), f =  \phi_{(k)}^2$ and parameter values $\kappa = \lambda_{q+1} r $, $\zeta = \ell^{-5}$, $C_a = \ell^{-9}$, and $L = 3$. We are justified in these choices because $\zeta^{3} = \ell^{-15} = \lambda_{q+1}^{15\eta} \leq \lambda_{q+1}^{\frac 14}$. Applying Lemma~\ref{commutator estimate} and \eqref{linear estimate lemma} yields
\begin{align}
\norm{E_1^B}_{L^1} 
\lesssim \norm{E_1^B}_{L^p} 
&\leq \sum_{ k \in \Lambda_B} 
\norm{\mathcal{R}^B \left( \nabla (a_{(k)}^2)\mathbb{P}_{\geq (\lambda_{q+1} r/2)}(\phi_{(k)}^2) ( k_1 \otimes k_2 -k_2 \otimes k_1 )\right)}_{L^p} \notag \\
& \lesssim \ell^{-9} \lambda_{q+1}^{-1} r^{-1} \norm{\phi_{(k)}^2 ( k_1 \otimes k_2 -k_2 \otimes k_1 )}_{L^p} \notag \\
& \lesssim \ell^{-9} \lambda_{q+1}^{-1} r^{-1} \norm{\phi_{(k)}}_{L^{2p}}^2 \notag \\
& \lesssim \ell^{-9} \lambda_{q+1}^{-1} r^{-1} r^{\frac{1}{p} - 1} \notag \\
& \lesssim \ell^{-9} \lambda_{q+1}^{-1} r^{\frac{1}{p} -2} \, . 
\label{eq:Shaq}
\end{align}

For the second term in the decomposition of $\RR_{osc}^B$, namely $E_2^B$, we apply the product estimate \eqref{prod_est_highfreq},  along with the magnitude bounds \eqref{mag amp estimates} and  \eqref{vel amp estimates} to derive
\begin{align*}
\norm{E_2^B }_{L^1} 
&\leq  \sum_{ k \neq k' \in \Lambda_B} \norm{a_{(k)}a_{(k')}\phi_{(k)} \phi_{(k')} }_{L^1} + 
\sum_{ k \in \Lambda_u ,  k' \in \Lambda_B }\norm{a_{(k)} a_{(k')}  \phi_{(k)} \phi_{(k')} }_{L^1} \\ 
&\lesssim \sum_{ k \neq k' \in \Lambda_B }\norm{a_{(k)} }_{C^0}^2\norm{ \phi_{(k)} \phi_{(k')}}_{L^1} + 
\sum_{ k \in \Lambda_u, k' \in \Lambda_B  } \norm{a_{(k)}}_{C^0}\norm{a_{(k')}}_{C^0}   \norm{ \phi_{(k)} \phi_{(k')}}_{L^1}  \\
& \lesssim \ell^{-4}r \, .
\end{align*} 
Combining the estimates for $E_1^B$ and $E_2^B$ we conclude that
\begin{equation}
\label{magnetic oscillation estimate}
\norm{\mathring{R}_{osc}^B}_{L^1} \lesssim \ell^{-4}r + \ell^{-9} \lambda_{q+1}^{-1} r^{\frac{1}{p} -2} \lesssim \ell^{-9} \lambda_{q+1}^{-1} r^{\frac{1}{p} -2} 
\end{equation}
upon recalling that $r = \lambda_{q+1}^{-\frac 34}$, and that $p$ is close to $1$.

\subsubsection{Corrector Error}
Due to the smallness of the corrector terms, to estimate \eqref{corrector mag stress}, it suffices to simply apply Cauchy-Schwarz and use \eqref{lp mag corrector estimate}, \eqref{lp vel corrector estimate}, \eqref{mag perturbation l2 }, and \eqref{vel perturbation l2}:
\begin{align}
\label{magnetic corrector estimate}
\norm{\mathring{R}_{corr}^B }_{L^1} &\lesssim \norm{ w_{q+1}^c \otimes d_{q+1} }_{L^1} + \norm{ w_{q+1}^p \otimes d_{q+1}^c }_{L^1} \notag \\ 
&\lesssim \norm{w_{q+1}^c }_{L^2}\norm{d_{q+1} }_{L^2} +  \norm{ w_{q+1}^p}_{L^2} \norm{ d_{q+1}^c }_{L^2} \notag \\
& \lesssim \ell^{-12} \lambda_{q+1}^{-1} \delta_{q+1}^{\frac{1}{2}} + \ell^{-7}\lambda_{q+1}^{-1}\delta_{q+1}^{\frac{1}{2}}  \notag \\
& \lesssim \ell^{-12} \lambda_{q+1}^{-1} \delta_{q+1}^{\frac{1}{2}} \, .
\end{align}
This concludes the estimates necessary to bound $\RR_{osc}^B$.

\subsection{Estimates for the Reynolds stress}
\subsubsection{Linear Error}
To  bound \eqref{lin vel Reynolds} we proceed just as we did for \eqref{lin mag stress}. As we had for the magnetic perturbations, by \eqref{div free velocity} we have $\partial_t w_{q +1 } = \curl  \curl ( \sum_{ \Lambda_u} \partial_ta_{(k)} W_k^c  + \sum_{  \Lambda_B} \partial_ta_{(k)} W_k^c)$. Therefore we can obtain the same estimate as in \eqref{mag time derivative} except we account for the worse amplitudes estimates we get for  $k \in \Lambda_u $:
\begin{equation*}
\norm{ \mathcal{R} \partial_t w_{q+1} }_{L^1} \lesssim \sum_{k \in \Lambda_u}\norm{  a_{(k)} }_{C_{x,t}^2}\norm{ \nabla W_k^c }_{L^p} \lesssim \ell^{-22} \lambda_{q+1}^{-1} \,.
\end{equation*}
Furthermore, an examination of \eqref{mag cross terms} shows that the same bound will hold for cross terms in the Reynolds stress (again using the fact that $\lambda_{q+1}^{-1} \ll r \ll \ell $ ):
\begin{equation}
\label{vel cross terms}
\norm{ v_{\ell} \mathring{\otimes} w_{q+1} + w_{q+ 1} \mathring{\otimes} v_{\ell} - B_{\ell} \mathring{\otimes} d_{q+1} - d_{q+1} \mathring{\otimes} B_{\ell} }_{L^1} \lesssim \ell^{-4} r^{\frac{1}{2}} \,. 
\end{equation} 
Therefore we obtain the same bound for the linear Reynolds stress as we had obtained earlier for the linear magnetic stress in \eqref{magnetic linear estimate}:
\begin{equation}
\label{Reynolds linear estimate}
\norm{ \mathring{R}_{lin}^u }_{L^1} \leq \norm{ \mathcal{R} \partial_t w_{q+1} }_{L^1} +  \norm{ v_{\ell} \mathring{\otimes} w_{q+1} + w_{q+ 1} \mathring{\otimes} v_{\ell} - B_{\ell} \mathring{\otimes} d_{q+1} - d_{q+1} \mathring{\otimes} B_{\ell} }_{L^1} \lesssim \ell^{-4} r^{\frac{1}{2}} \,.
\end{equation}

\subsubsection{Oscillation error} 

In order to estimate \eqref{vel osc stress} we decompose it into three terms:
\begin{align*}
    \mathring{R}_{osc}^u  &= E_{1,1}^u + E_{1,2}^u + E_2^u \, ,
\end{align*}
where 
\begin{align*}
    E_{1,1}^u &:= \sum_{k \in \Lambda_u} \mathcal{R}\left(  \nabla (a_{(k)}^2) \mathbb{P}_{\neq 0}(\phi_{(k)}^2)k_1 \otimes k_1 \right) \\
    E_{1,2}^u &:= \sum_{k \in \Lambda_B}\mathcal{R} \left( \nabla(a_{(k)}^2) \mathbb{P}_{\neq 0}(\phi_{(k)}^2)(k_1 \otimes k_1 - k_2 \otimes k_2) \right) 
\end{align*}
and $E_2^u$ is defined by the high frequency terms on the last two lines on the right side of \eqref{vel osc stress}. 

To estimate $E_{1,1}^u$ we again apply Lemma~\ref{commutator estimate} with  the same parameters,  except now $C_a = \ell^{-14}$ and $\zeta = \ell^{-10}$. This leads to
\begin{align*}
\norm{E_{1,1}^u }_{L^1}  \lesssim  \sum_{k \in \Lambda_u} \norm{\mathcal{R}\left(  \nabla (a_{(k)}^2) \mathbb{P}_{\neq 0}(\phi_{(k)}^2)k_1 \otimes k_1 \right) }_{L^p} 
                 & \lesssim \ell^{-14} \lambda_{q+1}^{-1} r^{-1}  \norm{ (\phi_{(k)}^2)k_1 \otimes k_1 }_{L^p}  \\
                 & \lesssim \ell^{-14} \lambda_{q+1}^{-1} r^{-1} \|\phi_{(k)}\|_{L^{2p}}^2\\
                 & \lesssim \ell^{-14} \lambda_{q+1}^{-1}  r^{\frac{1}{p} -2} \, .
\end{align*}
For $E_{1,2}^u$, we can just use the estimate for $E_1^B$ since only the direction is different, and $\RSZ^B$ obeys the same bounds as $\RSZ$. From \eqref{eq:Shaq} we then obtain
\begin{equation*}
\norm{E_{1,2}^u }_{L^1} \lesssim \ell^{-9} \lambda_{q+1}^{-1} r^{\frac{1}{p} -2} \,.
\end{equation*}

Lastly we bound the $E_2^u$ stress given by the last two lines on the right side of \eqref{vel osc stress}.
Since we need only consider the amplitude functions themselves and not their derivatives, setting $j=0$ in  \eqref{mag amp estimates} and \eqref{vel amp estimates} we arrive at
\begin{align*}
\norm{E_2^u }_{L^1} &\leq \Biggl( \sum_{ k \neq k' \in \Lambda_u }   +   \sum_{ k \neq k' \in \Lambda_B }   +\sum_{ k \in \Lambda_u, k' \in \Lambda_B }\Biggr) \norm{a_{(k)}a_{(k')}\phi_{(k)} \phi_{(k')} }_{L^1}   \\
   &\lesssim \ell^{-4}r.
\end{align*} 
Combining the estimates obtained for the three parts in which we have decomposed $\RR_{osc}^u$ we obtain
\begin{equation}
\label{Reynolds oscillation estimate}
\norm{\mathring{R}_{osc}^u }_{L^1} \lesssim \ell^{-4}r + \ell^{-9} \lambda_{q+1}^{-1} r^{\frac{1}{p} -2} + \ell^{-14} \lambda_{q+1}^{-1}r^{\frac{1}{p} -2} \lesssim \ell^{-14} \lambda_{q+1}^{-1}r^{\frac{1}{p} -2}\,.
\end{equation}

\subsubsection{Corrector Error}
First, note that by inspection one can check that \eqref{corrector vel stress} can be written in the following  symmetric way as 
\begin{equation*}
w_{q+1}^c \mathring{\otimes} w_{q+1}^c + w_{q+1}^p \mathring{\otimes} w_{q+1}^c + w_{q+1}^c \mathring{\otimes} w_{q+1}^p - (d_{q+1}^c \mathring{\otimes} d_{q+1}^c + d_{q+1}^p \mathring{\otimes} d_{q+1}^c + d_{q+1}^c \mathring{\otimes} d_{q+1}^p) \,. 
\end{equation*}
We now proceed  to estimate $\mathring{R}_{corr}^u$ as we did for $\mathring{R}_{corr}^B$:
\begin{align}
\label{Reynolds corrector estimate}
\norm{\mathring{R}_{corr}^u }_{L^1} &\leq \norm{ w_{q+1} \mathring{\otimes} w_{q+1}^c }_{L^1} + \norm{w_{q+1}^c \mathring{\otimes} w_{q+1}^p}_{L^1} + \norm{d_{q+1} \mathring{\otimes} d_{q+1}^c}_{L^1} + \norm{d_{q+1}^c \mathring{\otimes} d_{q+1}^p}_{L^1} \notag \\
&\lesssim  \norm{w_{q+1} }_{L^2} \norm{w_{q+1}^c }_{L^2} + \norm{w_{q+1}^c }_{L^2} \norm{ w_{q+1}^p }_{L^2} + \norm{ d_{q+1}}_{L^2} \norm{d_{q+1}^c }_{L^2} + \norm{d_{q+1}^c }_{L^2} \norm{ d_{q+1}^p }_{L^2} \notag\\
& \lesssim \delta_{q+1}^{\frac{1}{2}} \ell^{-12}\lambda_{q+1}^{-1} + \delta_{q+1}^{\frac{1}{2}} \ell^{-7}\lambda_{q+1}^{-1} \notag\\
& \lesssim \delta_{q+1}^{\frac{1}{2}}\ell^{-12}\lambda_{q+1}^{-1} \,.
\end{align}

\subsection{Verification of Inductive estimate for  magnetic and Reynolds Stress}

Finally, we verify \eqref{mag inductive assumptions} and \eqref{vel inductive assumptions} for the stresses. Using \eqref{magnetic linear estimate}, \eqref{magnetic oscillation estimate}, \eqref{magnetic corrector estimate}, and \eqref{mag commutator estimate} we have  
\begin{align*}
\|\mathring{R}_{q+1}^B \|_{L^1} &\leq \| \mathring{R}_{lin}^B \|_{L^1} +  \| \mathring{R}_{osc}^B\|_{L^1} +  \|\mathring{R}_{corr}^B \|_{L^1} +  \|\mathring{R}_{comm}^B \|_{L^1} \\
&\lesssim   \ell^{-4} r^{\frac{1}{2}} + \ell^{-9} \lambda_{q+1}^{-1} r^{\frac{1}{p} -2}  + \ell^{-12} \lambda_{q+1}^{-1} \delta_{q+1}^{\frac{1}{2}}  + \ell^2 \lambda_{q}^4 \\
&\leq \ell^{-5} r^{\frac{1}{2}} + \ell^{-10} \lambda_{q+1}^{-1} r^{\frac{1}{p} -2} + \ell^{-13} \lambda_{q+1}^{-1} + \ell^{\frac 32} \lambda_{q}^4\\
&\leq 2 (\ell^{-10} \lambda_{q+1}^{-1} r^{\frac{1}{p} -2} + \ell^{\frac 32} \lambda_{q}^4) \\
& \leq c_B \delta_{q+2} \, ,
\end{align*}
upon taking $p$   close to $1$, and $a$ sufficiently large to make the last inequality true.  Finally, we estimate the velocity Reynolds stress. From \eqref{Reynolds linear estimate}, \eqref{Reynolds oscillation estimate}, \eqref{Reynolds corrector estimate}, and \eqref{vel commutator estimate} we derive
\begin{align*}
\norm{\mathring{R}_{q+1}^u }_{L^1} &\leq \norm{ \mathring{R}_{lin}^u }_{L^1} +  \norm{ \mathring{R}_{osc}^u}_{L^1} +  \norm{\mathring{R}_{corr}^u }_{L^1} +  \norm{\mathring{R}_{comm}^u }_{L^1} \\
&\lesssim   \ell^{-4} r^{\frac{1}{2}}  + \ell^{-14} \lambda_{q+1}^{-1} r^{\frac{1}{p} -2}  + \ell^{-12} \lambda_{q+1}^{-1} \delta_{q+1}^{\frac{1}{2}}  + \ell^2 \lambda_{q}^4 \\
&\leq \ell^{-5} r^{\frac{1}{2}} + \ell^{-15} \lambda_{q+1}^{-1} r^{\frac{1}{p} -2} + \ell^{-13} \lambda_{q+1}^{-1} + \ell^{\frac{3}{2}} \lambda_{q}^4\\
&\leq 2 (\ell^{-15} \lambda_{q+1}^{-1} r^{\frac{1}{p} -2} + \ell^{\frac{3}{2}} \lambda_{q}^4) \\
& \leq c_u \delta_{q+2} \,,  
\end{align*}
as above. This concludes the proof of the main iteration in Proposition~\ref{prop:main:iteration}.

\section{Proof of Theorem~\ref{thm:main}}
\label{final proof}
Having established Proposition~\ref{prop:main:iteration}, we now turn to the proof of Theorem~\ref{thm:main}. Consider the mean-free, incompressible vector fields $u_0$ and $B_0$ given by 
\begin{equation}
\label{initial fields}
u_0 = \frac{t}{(2\pi)^{\frac{3}{2}}}(\sin( \lambda_0^{\frac{1}{2}} x_3), 0, 0   ) \quad B_0 = \frac{t}{(2\pi)^{3}}( \sin(\lambda_0^{\frac{1}{2}} x_3),  \cos(\lambda_0^{\frac{1}{2}} x_3 ), 0 ) \,.
\end{equation}
A calculation shows that $u_0 \cdot \nabla B_0 - B_0 \cdot \nabla u_0 = 0 $ and that $u_0 \cdot \nabla u_0 - B_0 \cdot \nabla B_0  = 0$. Therefore, $u_0$ and $B_0$ satisfy \eqref{relax mhd 1} and \eqref{relax mhd 2}
with 
\begin{equation*}
\mathring{R}_0^u = 
\frac{1}{\lambda_0^{\frac{1}{2}} (2 \pi)^{\frac{3}{2}}   }
\begin{bmatrix}
0 & 0 & -\cos(\lambda_0^{\frac{1}{2}} x_3 )\\
0 & 0 & 0 \\
-\cos(\lambda_0^{\frac{1}{2}} x_3)&  0 & 0 
\end{bmatrix}   
\end{equation*}
and 
\begin{equation*}
\mathring{R}_0^B = 
\frac{1}{\lambda_0^{\frac{1}{2}} (2 \pi)^3   }
\begin{bmatrix}
0 & 0 & -\cos(\lambda_0^{\frac{1}{2}} x_3 )\\
0 & 0 & \sin(\lambda_0^{\frac{1}{2}}x_3 ) \\
\cos(\lambda_0^{\frac{1}{2}} x_3)&  -\sin(\lambda_0^{\frac{1}{2}}x_3 ) & 0 
\end{bmatrix}.
\end{equation*} 
We have that $\norm{\mathring{R}_0^B }_{L^1}, \norm{ \mathring{R}_0^u }_{L^1} \leq \lambda_0^{-\frac{1}{2}} < \lambda_0^{-2b\beta} = \delta_1$. Therefore by taking $a$ sufficiently large we have $\norm{ \mathring{R}_0^u }_{L^1} \leq c_u \delta_1$ and $\norm{ \mathring{R}_0^B }_{L^1} \leq c_B\delta_1$.   Similarly we can show that the other conditions in \eqref{mag inductive assumptions} and \eqref{vel inductive assumptions} are all satisfied (possibly by taking $a$ larger). Therefore, we can apply Proposition~\ref{prop:main:iteration} to get the existence of a sequence of iterates $(u_{q+1}, \mathring{R}_{q+1}^{u}, B_{q+1}, \mathring{R}_{q+1}^B )$ which satisfy \eqref{relax mhd} and obey the bounds \eqref{mag inductive assumptions}--\eqref{main inductive assumption}.

By interpolation, we have for any $\beta' \in (0,\frac{\beta}{2 + \beta} )$ the sequence of velocity and magnetic increments is summable in $H^{\beta'}$, i.e.
\begin{align*}
    &\sum_{q \geq 0} \norm{ u_{q+1} - u_q }_{H^{\beta'}} +\sum_{q \geq 0} \norm{ B_{q+1} - B_q }_{H^{\beta'}} \\
    &\leq  \sum_{q \geq 0} \norm{ u_{q+1} - u_q }_{L^2}^{1- \beta'}\norm{ u_{q+1} - u_q }_{H^1}^{\beta'} + 
     \sum_{q \geq 0} \norm{ B_{q+1} - B_q }_{L^2}^{1- \beta'}\norm{ B_{q+1} - B_q }_{H^1}^{\beta'}\\
    &\lesssim \sum_{q \geq 0} \delta_{q+1}^{\frac{1-\beta'}{2}}\lambda_{q+1}^{2 \beta' } =  
    \sum_{ q \geq 0} \lambda_{q+1}^{-\beta(1 - \beta')  + 2\beta'  } \lesssim 1.
\end{align*}
The sequence $\{(u_q,B_q)\}_{q\geq 0}$ is hence Cauchy  and   we may define a limiting pair $(u ,B) = \lim_{q \to \infty} (u_q, B_q)$. This pair satisfies \eqref{eq:ideal:mhd} because $ \lim_{q \to \infty} \mathring{R}_{q}^u =   \lim_{q \to \infty} \mathring{R}_{q}^B = 0 $ in $C([0,1], L^1)$. Therefore, we have a weak solution of \eqref{eq:ideal:mhd} which lies in $C([0,1]; H^{\beta'})$, proving the first part of Theorem~\ref{thm:main} replacing $\beta$ by $\beta'$.
 
Now we will show that the magnetic helicity of the weak solution of \eqref{eq:ideal:mhd} at least doubles from time $0$ to time $1$, and is nonzero at time $1$. The vector field $B_0$ has associated with it the vector potential $A_0$:  
\begin{equation*}
A_0 =  \frac{t }{\lambda_{0}^{\frac{1}{2}} (2\pi)^3 }(  \sin(\lambda_0^{\frac{1}{2}} x_3),  \cos(\lambda_0^{\frac{1}{2}} x_3 ), 0   ) 
\end{equation*}
Therefore we can compute the first iterate of the magnetic helicity,  $\HH_{0,B,B}(t) := \int_{\mathbb{T}^3} A_0\cdot B_0 $, as
\begin{equation*}
\HH_{0,B,B}(t) = \int_{\mathbb{T}^3} A_0 \cdot B_0 dx = \frac{t^2}{(2 \pi)^6 \lambda_0^{\frac{1}{2}}} \int_{\mathbb{T}^3} |B_0|^2 dx = \frac{t^2}{(2 \pi)^3 \lambda_0^{\frac{1}{2}}} = \frac{t^2}{(2 \pi)^3 a^{\frac{1}{2}}}
\end{equation*}
Next, we wish to estimate the deviation between this quantity and  the magnetic helicity for the limiting vector field $B$:
\begin{equation*}
|\HH_{B,B}(t) - \HH_{0,B,B}(t)| = \left| \int_{\mathbb{T}^3} A \cdot B dx- \int_{\mathbb{T}^3} A_0 \cdot B_0dx \right| \leq \norm{A - A_0 }_{L^2} \norm{B}_{L^2} + \norm{ A_0 }_{L^2} \norm{ B- B_0}_{L^2} \, .
\end{equation*} 
Using \eqref{mag inductive assumptions} we have that $\norm{ B}_{L^2} \leq 1$ and by construction,  $A_0: \norm{ A_0}_{L^2} \leq  \lambda_0^{-\frac{1}{2}} (2\pi)^{-\frac{3}{2}}$. 
Therefore, we have   
\begin{align*}
|\HH_{B,B}(t) - \HH_{0,B,B}(t)| \leq \norm{A - A_0 }_{L^2} + \frac{1}{\lambda_0^{\frac{1}{2}} (2\pi)^{\frac{3}{2}} } \norm{ B- B_0 }_{L^2} \,.
\end{align*}
Applying the triangle inequality, \eqref{main inductive assumption}, and using the fact that $b > 2$ and consequently that $b^{q} \geq bq $  for $q \geq 1$ we can estimate $\|B- B_0 \|_{L^2}$ as
\begin{equation*}
\norm{B - B_0 }_{L^2} \leq \sum_{q \geq 0} \norm{ B_{q+1} - B_q }_{L^2} \leq   \sum_{q \geq 0} \delta_{q+1}^{\frac{1}{2}} =   \sum_{q \geq  1} (a^{b^q})^{-\beta} \leq \sum_{q \geq  1} (a^{bq})^{-\beta} = \frac{a^{-\beta b}}{1- a^{-\beta b} }.
\end{equation*}
To estimate $\norm{ A-A_0 }_{L^2}$ we use the fact that we can take $A_q$ to be divergence free  for all $q \in \mathbb{N}$. This choice allows us to recover $A_q$ using the Biot-Savart law:
\begin{equation*}
\norm{ A-A_0}_{L^2} \leq \sum_{q \geq 0} \norm{ A_{q+1} - A_q }_{L^2} \leq \sum_{q \geq 0} \norm{ \curl (-\Delta)^{-1}(B_{q+1} - B_q)   }_{L^2}.  
\end{equation*}
Now, recall that $B_{q+ 1} - B_q = d_{q+1} + B_{\ell} - B_{q}$, and therefore 
\begin{equation*}
\norm{ \curl (-\Delta)^{-1}(B_{q+1} - B_q) }_{L^2} \leq \norm{\curl (-\Delta)^{-1} d_{q+1}  }_{L^2} + \norm{ \curl (-\Delta)^{-1} (B_{\ell}  - B_q  )  }_{L^2}.
\end{equation*}
We first estimate the $\norm{ \curl (-\Delta)^{-1} (B_{\ell}  - B_q  )  }_{L^2}$ term. Note that $B_{\ell} - B_{q}$ has mean zero, and thus $ \norm{ \curl (-\Delta)^{-1} (B_{\ell}  - B_q  )  }_{L^2} \leq \norm{ B_{\ell}  - B_q }_{L^2}$. Furthermore, using standard mollification estimates and \eqref{mag inductive assumptions} we obtain the bound
\begin{equation*}
\norm{B_{q} - B_{\ell} }_{L^2} \leq (2\pi)^\frac{3}{2} \ell \lambda_{q}^2 = (2 \pi)^{\frac{3}{2}} \lambda_{q+1}^{-\eta + \frac{2}{b}}  \, .
\end{equation*}
Summing this expression gives 
\begin{equation*}
\sum_{q \geq 0}\|B_q - B_{\ell} \|_{L^2} \leq (2 \pi)^{\frac{3}{2}}\sum_{q \geq  0} \lambda_{q+1}^{-\eta + \frac{2}{b}} \leq (2 \pi)^{\frac{3}{2}}\sum_{q \geq 1} (a^{bq}) ^{-\eta + \frac{2}{b}} \leq (2 \pi)^{\frac{3}{2}} \sum_{q \geq 1} (a^{-1})^q = \frac{(2\pi)^{\frac{3}{2}}  a^{-1} } {1 - a^{-1}  }
\end{equation*}
where we used that $ \eta b \geq 3 $ and that $ b \geq 2$. 
Finally, we estimate  $ \sum_{q \geq 0} \norm{\curl (-\Delta)^{-1} d_{q+1}  }_{L^2}$.  
Recall that $d_{q+1} = \curl  \curl ( \sum_{k \in \Lambda_B} a_{(k)} D_k^c) $. Using that $\curl ( \sum_{k \in \Lambda_B} a_{(k)} D_k^c)$ is divergence free, we have that 
\begin{equation*}
    \curl (-\Delta)^{-1} \curl  \curl \Biggl( \sum_{k \in \Lambda_B} a_{(k)} D_k^c \Biggr) = \curl  \Biggl( \sum_{k \in \Lambda_B} a_{(k)} D_k^c   \Biggr) \, .
\end{equation*}
From the triangle inequality,  \eqref{mag amp estimates}, Lemma~\ref{linear estimate lemma}, and the fact that $\ell^{-1} \ll \lambda_{q+1}$ we have
\begin{align*}
     \sum_{k \in \Lambda_B} \norm{ \curl  (a_{(k)} D_k^c) }_{L^2} 
    &\leq \sum_{k \in \Lambda_B} \norm{ \curl (a_{(k)}D_k^c) }_{L^2} \leq   \sum_{k \in \Lambda_B}  \norm{ \nabla a_{(k)} \times D_k^c }_{L^2} +  \norm{a_{k ,B} \curl D_k^c }_{L^2}  \notag \\
    &\leq \sum_{k \in \Lambda_B}  \norm{ \nabla a_{(k)} }_{C^0} \norm{ D_{k}^c }_{L^2}  + \norm{ a_{(k)}  }_{C^0} \norm{\curl D_k^c }_{L^2}  \notag\\
    & \lesssim \ell^{-7}\lambda_{q+1}^{-2} + \ell^{-2}\lambda_{q+1}^{-1}  
    \lesssim \ell^{-2}\lambda_{q+1}^{-1}.
\end{align*}
Using a factor of $\ell$ to absorb the implicit constant, and using that $\ell^{-3}\lambda_{q+1}^{-1} \leq \lambda_{q+1}^{-\frac{1}{2}}$,  we deduce that 
\begin{equation*}
  \sum_{q \geq 0} \norm{\curl (-\Delta)^{-1} d_{q+1}  }_{L^2}  \leq \sum_{q \geq 0}\lambda_{q+1}^{-\frac{1}{2}} 
  \leq \sum_{q \geq 1} (a^{bq})^{-\frac{1}{2}} = \frac{a^{-\frac{b}{2} }}{ 1 - a^{-\frac{b}{2}}  } \leq \frac{a^{-1}  }{1 - a^{-1} } 
\end{equation*}
where we used that $b \geq 2$.

Therefore we have proven that 
\begin{equation*}
\norm{ A-A_0}_{L^2} \leq \frac{2 (2\pi)^{\frac{3}{2}} a^{-1}  }{1 - a^{-1} }
\end{equation*}
which gives the bound 
\begin{equation*}
|\HH_{B,B}(t) - \HH_{0,B,B}(t)| \leq \frac{2 (2\pi)^{\frac{3}{2}} a^{-1}  }{1 - a^{-1} } +  \frac{a^{-\beta b}}{1- a^{-\beta b} }\frac{1}{a^{\frac{1}{2}} (2\pi)^{\frac{3}{2}} } = \frac{1}{a^{\frac{1}{2}}} \left( \frac{2 (2\pi)^{\frac{3}{2}} a^{-\frac{1}{2}}  }{1 - a^{-1} }    +  \frac{a^{-\beta b}}{1- a^{-\beta b} }\frac{1}{ (2\pi)^{\frac{3}{2}} } \right).
\end{equation*}
Since $\mathcal{H}_{0,B,B}(1) = \frac{1}{(2 \pi)^3 a^{\frac{1}{2}}}$ by taking $a$ sufficiently large, we can ensure that $ |\mathcal{H}_{B,B}(t) - \mathcal{H}_{0,B,B}(t)| \leq \frac{1}{3}\mathcal{H}_{0,B,B}(1)$. This implies that $\mathcal{H}_{B,B}(1) \geq \frac{2}{3}\mathcal{H}_{0,B,B}(1) > 0$ and $|\mathcal{H}_{B,B}(0)| \leq \frac{1}{3}\mathcal{H}_{0,B,B}(1)$, since $\mathcal{H}_{0,B,B}(0) = 0$. This shows that the magnetic helicity at least doubles in magnitude and is nonzero at time $1$.  
 
\begin{remark}
The total energy $\EE(t)$ and the cross-helicity $\mathcal{H}_{\omega,B}(t)$ can similarly be shown to not be conserved for the limiting solution $(u,B)$, if we use \eqref{initial fields} as the first term in the sequence. By inspection, the initial fields have non-trivial energy and cross-helicity which allows us to prove the non-conservation as in the proof of Theorem~\ref{thm:main}.   
\end{remark}

\appendix
\section{Proof of Geometric Lemmas}
\label{Proof of Geometric lemmas}

In this section, we will provide proofs of Lemmas~\ref{geometric lem 1} and~\ref{geometric lem 2}, following the classical arguments of~\cite{DLSZ13,BDLISZ15}.

\begin{proof}[Proof of Lemma~\ref{geometric lem 1}]
Let $\Lambda_B = \{ e_1, e_2, e_3, \frac{3}{5}e_1 + \frac{4}{5}e_2, -\frac{4}{5}e_2 - \frac{3}{5}e_3 \}$ and to these vectors, consider the orthonormal bases given by 
\begin{center}
  \begin{tabular}{ | l | c | r | }   
    $k$ & $k_1$ & $k_2$ \\ \hline
    $e_1 $&  $e_2$ & $e_3$ \\ 
    $e_2$ & $e_3$ & $e_1$ \\
    $ e_3$ & $e_1$ & $e_2$\\
    $ \frac{3}{5}e_1 + \frac{4}{5}e_2$ & $\frac{4}{5}e_1 - \frac{3}{5}e_2$ & $e_3$\\
    $ -\frac{4}{5}e_2 -\frac{3}{5}e_3$ & $\frac{3}{5}e_2 - \frac{4}{5}e_3$ & $e_1$\\   
  \end{tabular}
\end{center}
We define 
\begin{align*}
A_1 &:= e_2 \otimes e_3 - e_3 \otimes e_2, \quad 
A_2 := e_3 \otimes e_1 - e_1 \otimes e_3, \quad
A_3 := e_1 \otimes e_2 - e_2 \otimes e_1,\\
A_4 &:= \left(\frac{4}{5}e_1 - \frac{3}{5}e_2 \right) \otimes e_3 - e_3 \otimes \left(\frac{4}{5}e_1 - \frac{3}{5}e_2 \right), \quad 
A_5 := \left(\frac{3}{5}e_2 - \frac{4}{5}e_3 \right) \otimes e_1 - e_1 \otimes \left(\frac{3}{5}e_2 - \frac{4}{5}e_3 \right) \,.
\end{align*}
Using these matrices we can write 
\begin{equation}
\label{asy zero iden}
\frac{7}{4}A_1 + \frac{11}{3}A_2 + A_3 + \frac{35}{12}A_4 + \frac{5}{3}A_5 = 0 \,. 
\end{equation}
Since $A_1, A_2, A_3$ form a basis for the 3 $\times$ 3 skew-symmetric matrices, we can express any skew-symmetric matrix $A$ as a unique linear combination $A = c_1A_1 + c_2A_2 + c_3A_3$. Combining this with  \eqref{asy zero iden} gives 
\begin{equation*}
\left(\frac{7}{4} + c_1 \right)A_1 + \left(\frac{11}{3} + c_2 \right)A_2 + \left(1 + c_3 \right)A_3 + \frac{35}{12}A_4 + \frac{5}{3}A_5 = A \,.
\end{equation*}
Therefore we can define 
\begin{align*}
\gamma_{1,B} = \sqrt{\frac{7}{4} + c_1}, \quad
\gamma_{2,B}  = \sqrt{\frac{11}{3} + c_2}, \quad
\gamma_{3,B} = \sqrt{ 1 + c_3},  \quad 
\gamma_{4,B} = \sqrt{\frac{35}{12}},   \quad
\gamma_{5,B} = \sqrt{\frac{5}{3}} \, .
\end{align*}
For $\varepsilon_B < \sqrt{2}$, the $\gamma_i$ will be smooth. Therefore it suffices to take $\varepsilon_B = 1$.
\end{proof}

 \begin{proof}[Proof of Lemma~\ref{geometric lem 2}]
Proceeding as before let $\Lambda_u = \{ \frac{5}{13}e_1 \pm \frac{12}{13}e_2, \frac{12}{13}e_1 \pm \frac{5}{13}e_3, \frac{5}{13}e_2 \pm \frac{12}{13}e_3 \}$ and to these vectors, consider the orthonormal bases given by 
\label{velocitywave}
\begin{center}
  \begin{tabular}{ | l | c | r | }  
    $k$ & $k_1$ & $k_2$ \\ \hline
    $\frac{12}{13}e_1 \pm \frac{5}{13}e_2 $ &  $\frac{5}{13}e_1 \mp \frac{12}{13}e_2$ & $e_3$ \\ 
    $\frac{5}{13}e_1 \pm \frac{12}{13}e_3$ & $\frac{12}{13}e_1 \mp \frac{5}{13}e_3 $ & $e_2$ \\
    $\frac{12}{13}e_2 \pm \frac{5}{13}e_3$ & $\frac{5}{13}e_2 \mp \frac{12}{13}e_3$ & $e_1$\\
  \end{tabular}
\end{center} 
Note that $\Lambda_u \cap \Lambda_B = \emptyset$. Next, note that $\sum_{k\in\Lambda_u} \frac 12 k_1\otimes k_1 = \Id$, and thus by the implicit function theorem, there exists $\varepsilon_u$ such that for $S \in B_{\varepsilon_u}(\Id)$, $S$ can be expressed as a linear combination of the $S_i$ with positive coefficients. See~\cite{DLSZ13,BDLISZ15} for further details.
\end{proof}

\section{Proof of Magnetic Helicity Conservation}
\label{app:Taylor}

In this appendix we give the proof of Theorem~\ref{thm:Taylor}. For $ u, B \in L^3(0,T;L^3(\T^3))$ we have magnetic helicity conservation for \eqref{eq:ideal:mhd}, as in \cite{KangLee07}. A simple modification of this argument shows that Leray-Hopf solutions of \eqref{eq:MHD} satisfy a {\em magnetic helicity  balance} (by interpolation we have that $u, B \in L_{x,t}^{\frac{10}{3}}(\mathbb{T}^3) $):
\begin{equation}
\label{helicity balance}
    \int_{\mathbb{T}^3 } A \cdot B(t)dx +  2 \mu \int_0^t \int_{\mathbb{T}^3} \curl B \cdot B(s)  dx ds = \int_{\mathbb{T}^3} A \cdot B (0) dx.
\end{equation}

Assume that $(u_j, B_j)$ is a weak ideal sequence and  that $\mu_j \to 0$. Using  the uniform bounds coming from the {\em total energy inequality} \eqref{eq:energy:inequality} we have that 
\begin{equation*}
     \mu_j \int_0^t \int_{\mathbb{T}^3} |\curl B_j \cdot B_j| dx ds \leq t\mu_j^{\frac{1}{2}} \| \mu_j^{\frac{1}{2}} (\curl B_j) \|_{L_t^{\infty}L_x^2} \| B_j\|_{L_t^{\infty}L_x^2 } \to 0 \quad \text{ as } j \to \infty \, .
\end{equation*}
Therefore, passing to the limit in \eqref{helicity balance} (for $A_j$ and $B_j$), we obtain
\begin{equation*}
    \liminf_{j \to \infty} \int_{\mathbb{T}^3 } A_j \cdot B_j (t)dx = \liminf_{j \to \infty} \int_{\mathbb{T}^3 } A_j \cdot B_j (0)dx = \int_{\mathbb{T}^3} A \cdot B(0)dx
\end{equation*}
where the last equality comes from the fact that since $B_j(0) \rightharpoonup B(0)$ in $L^2$, $A_j(0) \to A(0)$ in $L^2$ and the product of a weakly convergent sequence and a strongly convergent sequence converges. 
By Aubin-Lions Lemma with the triple $L^2 \subset H^{-\frac{1}{2}} \subset H^{-3}$  applied to $B_j$, we conclude that $B_j(t)$ has a strongly convergent subsequence in $C([0,T]; H^{-\frac{1}{2}})$ (also denoted $B_j(t)$).  This implies $A_j(t)$ is strongly convergent in $C([0,T]; \dot{H}^{\frac{1}{2}})$. Along this subsequence
\begin{equation*}
    \int_{\mathbb{T}^3} A_j \cdot B_j(t) dx =  \int_{\mathbb{T}^3} |\nabla|^{\frac{1}{2}}A_j \cdot |\nabla|^{-\frac{1}{2}} B_j(t) dx \to  \int_{\mathbb{T}^3} |\nabla|^{\frac{1}{2}}A\cdot |\nabla|^{-\frac{1}{2}} B(t) dx =  \int_{\mathbb{T}^3} A \cdot  B(t) dx
\end{equation*}
where we are using that limit of the strongly convergent subsequence must coincide with the weak ideal limit by uniqueness of weak-* limits. Furthermore, we can extend this to the entire sequence to conclude  
\begin{equation*}
    \int_{\mathbb{T}^3} A(t) \cdot B(t) dx = \int_{\mathbb{T}^3} A(0) \cdot B(0) dx
\end{equation*}
as desired.

\section*{Acknowledgements}
R.B. was supported by an NSF Graduate Fellowship Grant No. 1839302. T.B. was supported by the NSF grant DMS-1600868. V.V. was supported by the NSF CAREER grant DMS-1652134. The authors are grateful to Theodore Drivas for pointing us to a number of references on magneto-hydrodynamic turbulence.

\end{document}